\documentclass[10pt]{amsart}

\usepackage{amsmath, amsthm, amssymb}
\usepackage{mathrsfs}
\usepackage[colorlinks]{hyperref}
\usepackage{tikz,pgf,tikz-3dplot}
\usepackage{subcaption}
\usepackage{multirow}
\usepackage{dynkin-diagrams}

\usetikzlibrary{shapes.geometric}
\usetikzlibrary{backgrounds}

\DeclareMathOperator{\conv}{conv}
\DeclareMathOperator{\Cayley}{Cayley}
\DeclareMathOperator{\Vor}{Vor}
\DeclareMathOperator{\vol}{vol}
\DeclareMathOperator{\wt}{wt}
\DeclareMathOperator{\supp}{supp}
\DeclareMathOperator{\Adjt}{Ad}
\DeclareMathOperator{\adjt}{ad}
\DeclareMathOperator{\tr}{tr}
\DeclareMathOperator{\ch}{ch}

\newcommand{\ZZ}{\ensuremath{\mathbb{Z}}}
\newcommand{\CUT}{\operatorname{CUT}}
\newcommand{\ad}{\operatorname{ad}}

\newtheorem{defin}{Definition}[section]

\newtheorem{theorem}[defin]{Theorem}
\newtheorem{proposition}[defin]{Proposition}
\newtheorem{lemma}[defin]{Lemma}
\newtheorem{example}[defin]{Example}
\newtheorem{corollary}[defin]{Corollary}
\newtheorem{definition}[defin]{Definition}

\newtheorem{remark}[defin]{Remark}
\newcommand\parapoint{\refstepcounter{defin}\smallskip\textbf{\thedefin.} }

\newcommand{\fonc}[5]{\begin{array}{ccccc}
#1 & : & #2 & \to & #3 \\
 & & #4 & \mapsto & #5 \\
\end{array}}

\begin{document}

\title{Coloring the Voronoi tessellation of lattices}

\date{March 22, 2021}

\author{Mathieu Dutour Sikiri\'c}
\address{M.~Dutour Sikiri\'c, Rudjer Boskovi\'c Institute, Bijeni\v cka
  54, 10000 Zagreb, Croatia}
\email{mathieu.dutour@gmail.com}

\author{David A. Madore}
\address{D.~Madore, T\'el\'ecom Paris / LTCI,
 D\'epartement INFRES, 19 place Marguerite Perey,
 91120 Palaiseau, France}
\email{david+math@madore.org}

\author{Philippe Moustrou}
\address{P.~Moustrou, Department of Mathematics and Statistics, UiT The Arctic University of Norway, N-9037 Troms{\o}, Norway}
\email{philippe.moustrou@uit.no}

\author{Frank Vallentin}
\address{F.~Vallentin, Department Mathematik/Informatik, Abteilung Mathematik, Universit\"at zu
  K\"oln, Weyertal~86--90, 50931 K\"oln, Germany}
\email{frank.vallentin@uni-koeln.de}

\begin{abstract}
  In this paper we define the chromatic number of a lattice: It is the
  least number of colors one needs to color the interiors of the cells
  of the Voronoi tessellation of a lattice so that no two cells
  sharing a facet are of the same color.

  We compute the chromatic number of the root lattices, their duals,
  and of the Leech lattice, we consider the chromatic number of
  lattices of Voronoi's first kind, and we investigate the asymptotic
  behaviour of the chromatic number of lattices when the dimension
  tends to infinity.

  We introduce a spectral lower bound for the chromatic number of
  lattices in spirit of Hoffman’s bound for finite graphs.  We compute
  this bound for the root lattices and relate it to the character
  theory of the corresponding Lie groups.
\end{abstract}



\maketitle

\markboth{M.~Dutour Sikiri\'c, D.~Madore, P.~Moustrou, and F.~Vallentin}{Coloring the Voronoi tessellation of lattices}

\tableofcontents

\section{Introduction}
\label{sec:introduction}

Let $\Lambda \subseteq \mathbb{R}^n$ be an $n$-dimensional lattice in
$n$-dimensional Euclidean space. One can tessellate space by lattice
translates of the lattice' Voronoi cell, which is defined as
\[
V(\Lambda) = \{x \in \mathbb{R}^n : \|x\| \leq \|x - v\| \text{ for
  all } v \in \Lambda \}.
\]
By $V(\Lambda)^\circ$ we denote the topological interior of
$V(\Lambda)$.  Now we consider translates $v + V(\Lambda)^\circ$, with
$v \in L$, as colored tiles of an $n$-dimensional mosaic in which one
has infinitesimal small interstices between the mosaic tiles. How many
colors does one need at least to get a colorful mosaic? In a colorful
mosaic two neighboring tiles receive different colors. This defines
the \emph{chromatic number} $\chi(\Lambda)$ of the lattice.

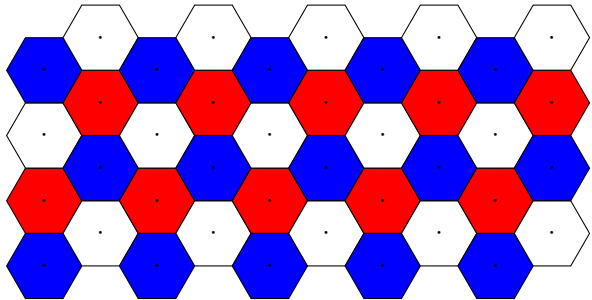
\begin{figure}[h]
\begin{center}
\begin{tikzpicture}[x=7.5mm,y=4.34mm]
  % some styles
  \tikzset{
    box/.style={
      regular polygon,
      regular polygon sides=6,
      color=black,
      minimum size=10mm,
      inner sep=0mm,
      outer sep=0mm,
      rotate=0,
    draw
    }
  }

\foreach \i in {0,...,4} {
 \node[box, fill = blue] at (2*\i,0) {$\cdot$};
 \node[box, fill = red] at (2*\i,2) {$\cdot$};
 \node[box, fill = white] at (2*\i,4) {$\cdot$};
 \node[box, fill = blue] at (2*\i,6) {$\cdot$};
 \node[box, fill = white] at (2*\i+1,1) {$\cdot$};
 \node[box, fill = blue] at (2*\i+1,3) {$\cdot$};
 \node[box, fill = red] at (2*\i+1,5) {$\cdot$};
 \node[box, fill = white] at (2*\i+1,7) {$\cdot$};
}
\end{tikzpicture}
\caption{Optimal coloring of the hexagonal lattice, $\chi(\mathsf{A}_2) = 3$.}
\end{center}
\end{figure}

More formally, we can also define the chromatic number of a lattice in
graph theoretical terms: Two distinct lattice translates of Voronoi
cells $v + V(\Lambda)$ and $w + V(\Lambda)$, with $v \neq w$, are
defining neighboring tiles whenever they share a facet,
i.e.\ their intersection is a polytope of maximal dimension $n-1$. The
differences $v - w$ are called \emph{strict Voronoi vectors} and the
set of these vectors is denoted by $\Vor(\Lambda)$.  Now the chromatic
number of $\Lambda$ equals the chromatic number of the Cayley graph on
the additive group $\Lambda$ with generating set $\Vor(\Lambda)$:
\[
\chi(\Lambda) = \chi(\Cayley(\Lambda, \Vor(\Lambda))).
\]
Here, the set of vertices of the Cayley graph are all elements of
$\Lambda$ and two vertices $v, w$ are adjacent whenever the difference
$v-w$ lies in the set of strict Voronoi vectors $\Vor(\Lambda)$. Note
that the Cayley graph is an $r$-regular infinite graph with
$r = |\Vor(\Lambda)|$.

\smallskip

In Section~\ref{ssec:lattices-voronoi-cells-voronoi-vectors} we recall
all the definitions and properties of lattices, their Voronoi cells, and
the Voronoi vectors, which we need later.

\medskip

The chromatic number of a lattice seems to be a natural parameter.
However, to the best of the authors' knowledge, see
\cite{MathOverflowQuestion} and \cite{EllenbergMathOverflowQuestion},
$\chi(\Lambda)$ has not been considered before.  The aim of this paper
is to start a systematic investigation of it.  Then, the following
questions immediately come to mind:

\subsection{Determination of the chromatic number}

\textit{What is the chromatic number of some interesting lattices? How to find
lower and upper bounds? Is there an algorithm to determine
$\chi(\Lambda)$ for a given lattice $\Lambda$?}

For instance, it is obvious that the chromatic number of the integer
lattice $\mathbb{Z}^n$ is two, an optimal coloring is given by the
black/white checkerboard pattern; see also
Theorem~\ref{thm:OrthogonalSum}.

We discuss simple lower and upper bounds for the chromatic number of a
general lattice in Section~\ref{ssec:easy-upper-bounds} and in
Section~\ref{ssec:easy-lower-bounds}. For instance, we show that
$\chi(\Lambda)$ is at most~$2^n$.

All two- and three-dimensional lattices are of Voronoi's first
kind. We consider the chromatic number of this class of lattices in
Section~\ref{sec:firstkind} where we compute the chromatic number of
all three-dimensional lattices. It would be interesting to have a
better understanding of the chromatic number of this class of lattices.

One of the most important classes of lattices are the root
lattices. We recall the definitions and classification in
Section~\ref{ssec:rootlattices}. One main result of our paper is the
determination of the chromatic number of all root lattices and their
duals. Table~\ref{table:ChromaticRecap} summarizes our results.

Note that we currently do not know the numerical value of
$\chi(\mathsf{D}_n)$. We only know it is equal to the chromatic number
of the (finite) vertex-edge-graph of the half-cube polytope
\[
\frac{1}{2} H_n = \conv\left\{x \in \{0,1\}^n : \sum_{k=1}^n x_k \text{ is even}\right\},
\]
which at the moment is only known up to dimension $n = 9$: For $n=4$,
$5$, $6$, $7$, $8$ and $9$ we have $\chi(\frac{1}{2}H_n) = 4$, $8$, $8$, $8$,
$8$ and $13$ (see \cite{OstergardChromatic9}).

\begin{table}[h]
\begin{tabular}{ccl}
\textit{lattice} & \textit{chromatic number} & \\
\hline
\\[-2ex]
$\mathbb{Z}^n$ & $2$ & Sec. \ref{sec:introduction}, Thm.~\ref{thm:OrthogonalSum}, Sec. \ref{sec:firstkind} \\
$\mathsf{A}_n$ & $n+1$ & Thm.~\ref{thm:An-and-An-star}, Thm. \ref{thm:lowera} \\
$\mathsf{A}_n^*$ & $n+1$ & Thm.~\ref{thm:An-and-An-star} \\
$\mathsf{D}_n$ & $\chi(\frac{1}{2}H_n)$ & Thm. \ref{thm:d} \\
$\mathsf{D}_n^*$ & $4$ & Sec. \ref{sec:introduction} \\
$\mathsf{E}_6$ & $9$ &  Sec. \ref{ssec:e6Schlafi}, Thm. \ref{thm:e}, Thm. \ref{Theorem_CritValues_E6}\\
$\mathsf{E}_6^*$ & $16$ & Thm. \ref{thm:chromaticdual} \\
$\mathsf{E}_7$ & $14$ &  Sec. \ref{ssec:e7e8sos}, Thm. \ref{thm:e}, Thm. \ref{Theorem_CritValues_E7} \\
$\mathsf{E}_7^*$ & $16$ & Thm. \ref{thm:chromaticdual} \\
$\mathsf{E}_8$ & $16$ & Sec. \ref{ssec:SpherePackingLeechE8}, Sec. \ref{ssec:e7e8sos}, Thm. \ref{thm:e}, Thm. \ref{Theorem_CritValues_E8} \\
$\Lambda_{24}$ & $4096$ & Sec. \ref{ssec:SpherePackingLeechE8} \\[1ex]
\end{tabular}
\caption{The chromatic number of important lattices, in particular the (irreducible) root lattices and their duals.}
\label{table:ChromaticRecap}
\end{table}

For the proof we use a generalization of a lower bound for the
chromatic number of finite graphs originally due to
Hoffman~\cite{Hoffman1970a}. Hoffman's bound is based on spectral
considerations: Let $A \in \mathbb{R}^{V \times V}$ be the adjacency
matrix of a finite graph $G = (V, E)$. Let $m(A)$ be the smallest
eigenvalue of $A$ and respectively let $M(A)$ be the largest
eigenvalue of $A$, then
\[
\chi(G) \geq 1 - \frac{M(A)}{m(A)}.
\]
Bachoc, DeCorte, Oliveira, and Vallentin~\cite{Bachoc2014a} showed how
to generalize the spectral bound (and its weighted variant due to
Lov\'asz \cite{Lovasz1979a}) from finite to infinite graphs. In
Section~\ref{ssec:setup} we review this generalization and specialize
it to $\chi(\Lambda)$. Here, classical Fourier analysis is
used. We show in Corollary~\ref{cor:Hoffman} that
\[
\chi(\Lambda) \geq 1 - \left(\inf\limits_{x \in \mathbb{R}^n}
  \frac{1}{|\Vor(\Lambda)|} \sum_{u \in \Vor(\Lambda)} e^{2\pi i u
    \cdot x}\right)^{-1},
\]
holds.

In Section~\ref{sec:spectralbound} we compute this bound for all
irreducible root lattices. Surprisingly, the result of this
computation can already be found in an Oberwolfach report by
Serre~\cite{Serre2004a} albeit in a different language and with a
different motivation. In his report Serre computed all critical values
of the characters of the adjoint representation of compact Lie
groups. However, the report does not contain proofs. In
Section~\ref{ssec:spectralbound-root} we provide proofs for the easy
cases $\mathsf{A}_n$ and $\mathsf{D}_n$. The cases $\mathsf{E}_6$,
$\mathsf{E}_7$, and $\mathsf{E}_8$ are much harder and we give Serre's
proof in Appendix~\ref{sec:serresproof} after recalling relevant facts
about compact Lie groups in Appendix~\ref{sec:liegroups}. We sketch an
alternative, computational proof, which is based on optimization, in
particular using sum of squares for the cases $\mathsf{E}_7$ and
$\mathsf{E}_8$ at the end of
Section~\ref{ssec:spectralbound-root}. The case $\mathsf{E}_6$ is
easier and does not require computer assistance.

Then, in Section~\ref{sec:coloring-root}, we construct several efficient
colorings of irreducible root lattices.

It would be nice to know the chromatic number of more important
lattices. Following the book \cite{Conway1988a} by Conway and Sloane
the next candidates one should consider are the $12$-dimensional
Coxeter-Todd lattice $\mathsf{K}_{12}$ and the $16$-dimensional
Barnes-Wall lattice $\mathsf{BW}_{16}$.  We expect that the spectral
lower bound gives a close approximation to the chromatic number.

We show in Section~\ref{ssec:SpherePackingLeechE8} that the chromatic
number of the Leech lattice $\Lambda_{24}$ in $24$ dimensions is
$4096$. This is a consequence of the sphere packing optimality of
$\Lambda_{24}$. It would be nice to have an independent (spectral)
proof of this fact.

Going back to general lattices: At the moment we do not know whether
there is a finite algorithm to compute the chromatic number of a
lattice which is given for example by a basis. Determining the strict
Voronoi vectors---and thus the Cayley graph
$\Cayley(\Lambda, \Vor(\Lambda))$ and the Voronoi cell
$V(\Lambda)$---is possible by a finite algorithm, see for example
\cite{Dutour2009a}, although it can occupy exponential space (and
therefore needs exponential time). The theorem of de Bruijn and
Erd\H{os} \cite{deBruijn1951a} implies that the chromatic number of
$\Lambda$ is equal to the largest chromatic number of all finite
subgraphs of $\Lambda$. This shows that the decision problem: ``Is
$\chi(\Lambda) \leq k$?'' is at least semidecidable.

\subsection{Generic and extremal behaviour of the chromatic number}

\textit{What is $\chi(\Lambda)$ of a random $n$-dimensional lattice?
  How fast can $\chi(\Lambda)$ grow depending on the dimension $n$?}

In Section~\ref{ssec:asymptotic} we prove that the chromatic
number of a generic $n$-dimensional lattice grows exponentially with
the dimension. There we show that there are $n$-dimensional
lattices $\Lambda_n$ with
\[
\chi(\Lambda_n) \geq 2 \cdot 2^{(0.0990\ldots - o(1))n}.
\]
It would be very interesting to understand the extremal behaviour.

\section{Background and first observations}

\subsection{Lattices, Voronoi cells, Voronoi vectors}
\label{ssec:lattices-voronoi-cells-voronoi-vectors}

A \textit{lattice} $\Lambda$ is a discrete free $\mathbb{Z}$-module in
an $n$-dimensional Euclidean space. If its rank is strictly lower than
$n$, then $\Lambda$ also defines a lattice in its linear span over
$\mathbb{R}$. We implicitly identify these two lattices, and assume
for the following definitions that $\Lambda$ is a full-rank lattice in
$\mathbb{R}^n$. We denote by $\Lambda^*$ the \textit{dual lattice} of
$\Lambda$:
\[
\Lambda^* =\{ x \in \mathbb{R}^n:  x \cdot y  \in \mathbb{Z}  \text{
  for all $y\in \Lambda$} \},
\]
where $x \cdot y$ denotes the standard Euclidean scalar product
between $x$ and $y$.  A \textit{fundamental region} of $\Lambda$ is a
region $\mathcal{R}\subset \mathbb{R}^n$ such that for any
$u \neq v \in \Lambda$, the volume of
$(u + \mathcal{R}) \cap (v + \mathcal{R})$ is $0$, and
$\mathbb{R}^n = \bigcup_{v \in \Lambda} (v + \mathcal{R})$.  The
\textit{volume} $\vol(\mathbb{R}^n / \Lambda)$ of $\Lambda$ is defined
as the volume of any of its fundamental region.  A fundamental region
of particular interest is the \textit{Voronoi cell} of $\Lambda$:
\[
V(\Lambda) = \{x \in \mathbb{R}^n : \|x\| \leq \|x - v\| \text{ for
  all } v \in \Lambda \}.
\]
A vector $u \in \Lambda \setminus\{0\}$ is called a \emph{strict
  Voronoi vector}, or sometimes a ``relevant'' vector, if the
intersection~$(u + V(\Lambda)) \cap V(\Lambda)$ is a facet, a face of
dimension~\hbox{$n-1$}, of $V(\Lambda)$. By a well-known
characterization of Voronoi (see for example \cite[Chapter~21,
Theorem~10]{Conway1988a} or \cite{Conway1997a}) the set of these
vectors is
\begin{equation}
\label{eq:Voronoi-characterization}
\Vor(\Lambda) = \{u \in \Lambda \setminus\{0\}: \pm u \text{ only shortest vectors
  in } u + 2\Lambda\}.
\end{equation}
Now the \textit{chromatic number} of $\Lambda$ equals the chromatic number of
the Cayley graph on the additive group $\Lambda$ with generating set
$\Vor(\Lambda)$:
\[
\chi(\Lambda) = \chi(\Cayley(\Lambda, \Vor(\Lambda))).
\]
Here, the set of vertices of the Cayley graph are all elements of
$\Lambda$ and two vertices $v, w$ are adjacent whenever the difference
$v-w$ lies in the set of strict Voronoi vectors $\Vor(\Lambda)$.

\subsection{Simple upper bounds for the chromatic number}
\label{ssec:easy-upper-bounds}

One can color a lattice $\Lambda$ periodically by using translates of
one of its sublattices $\Lambda'$ which does not contain Voronoi
vectors. More precisely it is enough to color the vertices of the
graph $G = (V, E)$ with
\[
V = \Lambda/\Lambda' \quad \text{and} \quad
E = \{ \{v + \Lambda',w + \Lambda'\} : v-w+u \in \Vor(\Lambda) \text{ for some } u \in \Lambda' \}.
\]
This immediately gives the following upper bound on $\chi(\Lambda)$:
\begin{lemma} \label{lem:upperboundquotient}
  Let $\Lambda' \subset \Lambda$ be a sublattice of $\Lambda$ with
  $\Lambda' \cap \Vor(\Lambda) = \emptyset$. Then,
  $\chi(\Lambda)$ is at most $|\Lambda / \Lambda'|$.
\end{lemma}
Sometimes, we can improve this bound by
coloring the vertices of the graph $G = (V, E)$ greedily.
This shows, see for example \cite[Chapter V.1]{Bollobas1998a},
that
\begin{equation}\label{eq:upperbounddegree}
\chi(\Lambda) \leq \chi(G) \leq \Delta(G) + 1,
\end{equation}
where $\Delta(G)$ is the largest degree of a vertex in $G$.

Now we take $\Lambda' = 2\Lambda$. Lemma~\ref{lem:upperboundquotient}
implies that $\chi(\Lambda) \leq 2^n$. If the number of Voronoi
vectors is not maximal, if $|\Vor(\Lambda)| < 2(2^n-1)$, then we can
improve this bound by using \eqref{eq:upperbounddegree}:
\begin{lemma}
\label{lem:upperbound-voronoi-vectors}
The chromatic number of $\Lambda$ is at most $|\Vor(\Lambda)|/2+1$.
\end{lemma}

For generic lattices, the number of Voronoi vectors is $2(2^n-1)$, so
that Lemma~\ref{lem:upperbound-voronoi-vectors} also gives an upper
bound of $2^n$ for the chromatic number of $\Lambda$.

\subsection{Simple lower bounds for the chromatic number}
\label{ssec:easy-lower-bounds}

For a general graph $G$ one has $\chi(G) \geq \chi(H)$ for every
induced subgraph $H$ of $G$. In particular, when choosing $H$ to be a
largest complete subgraph of $G$, we have $\chi(G) \geq \omega(G)$,
where $\omega(G)$ is the clique number of $G$.

Canonical finite induced subgraphs of
$\Cayley(\Lambda, \Vor(\Lambda))$ are the vertex-edge graphs of
Delaunay polytopes of $\Lambda$. A \emph{Delaunay polytope} of the
lattice $\Lambda$ is defined as follows: Let $x$ be a vertex of the
Voronoi cell $V(\Lambda)$. Consider all vectors
$v_1, \ldots, v_m \in \Lambda$ so that $x$ is contained in all the
translates $v_1 + V(\Lambda), \ldots, v_m + V(\Lambda)$. Then, the
convex hull $P = \conv\{v_1, \ldots, v_m\}$ of $v_1, \ldots, v_m$ is a
Delaunay polytope of~$\Lambda$. Clearly, all edges of $P$ lie in
$\Vor(\Lambda)$.

\begin{lemma}
The chromatic number of a lattice $\Lambda$ is at least the chromatic
number of the vertex-edge graph of any Delaunay polytope of $\Lambda$.
\end{lemma}

\subsection{Root lattices and their duals}
\label{ssec:rootlattices}

One of the most important classes of lattices are the root lattices.
Assume that $\Lambda \subseteq \mathbb{R}^n$ is an even lattice,
i.e.~we have $v \cdot v \in 2\mathbb{Z}$ for all $v \in \Lambda$.
Lattice vectors $v \in \Lambda$ with $v \cdot v = 2$ are called
\emph{root vectors}, or simply \emph{roots}. A \emph{root lattice}
$\Lambda \subseteq \mathbb{R}^n$ is an even lattice which is spanned
by roots.  Root lattices have been classified by Witt in 1941, see for
example \cite[Section 1.4]{Ebeling1994a}, and they are orthogonal
direct sums of the irreducible root lattices $\mathsf{A}_n$,
$\mathsf{D}_n$, $\mathsf{E}_6$, $\mathsf{E}_7$, and $\mathsf{E}_8$.
The strict Voronoi vectors of root lattices are precisely the root
vectors~\cite[Chapter~21, \S3.A]{Conway1988a} (and in fact, this
condition that only the shortest nonzero vectors are relevant
characterizes root lattices, see \cite{RajanShende1996}).  Also, the
combinatorial description of the Voronoi cells of root lattices is
well known: it is described in more detail in~\cite{MoodyPatera1992}.
Here we recall the definitions of the irreducible root lattices and
their duals.

Living inside the hyperplane $\Pi := \{x \in \mathbb{R}^{n+1} :
\sum_{k=0}^n x_k = 0\}$ (the coordinates being here numbered $0$
through~$n$), the irreducible root lattice $\mathsf{A}_n$ is defined as
\[
\mathsf{A}_n = \left\{x \in \mathbb{Z}^{n+1} : \sum_{k=0}^n x_k = 0\right\}.
\]

The dual lattice $\mathsf{A}_n^*$ naturally lives in the vector space
$\Pi^*$ dual to~$\Pi$, which can be identified with the
quotient of $\mathbb{R}^{n+1}$ by the diagonal line $\{(t,\ldots,t) :
t\in\mathbb{R}\}$.  But we can identify $\Pi^*$ with $\Pi$ itself by
choosing the representatives $(x_0,\ldots,x_n)$ of $\mathbb{R}^{n+1}$
modulo the diagonal which belong to~$\Pi$.

For every $n \geq 4$, the irreducible root lattice $\mathsf{D}_n$ is defined as
\[
\mathsf{D}_n = \left\{ x \in \mathbb{Z}^n : \sum_{k=1}^n x_k \text{ is even}\right\}
\]
and its dual lattice $\mathsf{D}_n^*$ equals
\begin{equation*}
\mathsf{D}_n^* = \mathbb{Z}^n \cup \left((1/2, \ldots, 1/2) + \mathbb{Z}^n \right).
\end{equation*}

The root lattice $\mathsf{E}_8$ can be constructed as the union of two
translates of the lattice $\mathsf{D}_8$:
\begin{equation*}
\mathsf{E}_8 = \mathsf{D}_8 \cup ((1/2, \ldots, 1/2) +  \mathsf{D}_8).
\end{equation*}
The lattice $\mathsf{E}_8$ is \textit{unimodular},
i.e.~$\mathsf{E}_8^*=\mathsf{E}_8$.  The lattice $\mathsf{E}_7$
(resp. $\mathsf{E}_6$) can be defined as a $7$-dimensional
(resp. $6$-dimensional) sublattice of $\mathsf{E}_8$:
\[
\mathsf{E}_7 = \{ (x_1, \ldots, x_8) \in \mathsf{E}_8 : x_7=x_8 \}
\]
and
\[
\mathsf{E}_6 = \{ (x_1, \ldots, x_8) \in \mathsf{E}_8 : x_6=x_7=x_8 \}.
\]
Then, if we define
\[
u = \frac{1}{4} (1,1,1,1,1,1,-3,-3) \quad \text{and} \quad v =
\frac{1}{3}(0,-2,-2,1,1,1,1,0),
\]
the dual lattices of $\mathsf{E}_7$ and $\mathsf{E}_6$ are
\[
\mathsf{E}_7^* = \mathsf{E}_7 \cup (u + \mathsf{E}_7) \quad \text{and} \quad
\mathsf{E}_6^* = \mathsf{E}_6 \cup (v + \mathsf{E}_6) \cup (-v +
\mathsf{E}_6).
\]

\subsection{The chromatic number of an orthogonal sum of lattices}

Eichler \cite{Eichler1952a} showed that one can decompose every
lattice as a pairwise orthogonal sum of indecomposable lattices and
that this decomposition is unique up to permutation of the summand;
Kneser \cite{Kneser1954a} gave a constructive and much simpler proof
of Eichler's result.

We prove that the chromatic number of a lattice is the maximum of the
chromatic numbers of its orthogonal summands.

This reduces in particular the study of the chromatic number of root
lattices to the irreducible root lattices $\mathsf{A}_n$,
$\mathsf{D}_n$, $\mathsf{E}_6$, $\mathsf{E}_7$, and $\mathsf{E}_8$.

\begin{lemma}
\label{lem:VorVecSum}
Let $\Lambda \subseteq \mathbb{R}^n$ be a lattice which can be written
as the orthogonal direct sum of lattices $\Lambda_1, \ldots, \Lambda_m
\subseteq \mathbb{R}^n$:
\[
\Lambda = \Lambda_1 \perp \Lambda_2 \perp \ldots \perp \Lambda_m,
\quad \text{with } m \in \mathbb{N},
\]
so that every lattice vector $v \in \Lambda$ can be uniquely
decomposed as $v = v_1 + \cdots + v_m$ with $v_i \in \Lambda_i$ and
$v_i$ is orthogonal to $v_j$ whenever $i \neq j$. Then,
\[
\Vor(\Lambda) = \bigcup_{i=1}^m \Vor(\Lambda_i).
\]
\end{lemma}

\begin{proof}
  By induction, we may assume that
  $\Lambda = \Lambda_1 \perp \Lambda_2$.

  Every $v\in \Vor(\Lambda)$ we can write as $v=v_1+v_2$ with
  $v_i\in V_i$. If both $v_1$ and $v_2$ are non zero, then
  $w = v_1 - v_2$ is different from $\pm v$. It lies in $v+2\Lambda$
  and satisfies $\Vert w \Vert = \Vert v \Vert $, yielding a
  contradiction.  So $v=v_i$ for $i \in \{1,2\}$. In particular,
  $\pm v_i$ must be the only minimal vectors in $v_i + 2\Lambda_i$, so
  that $v\in \Vor(\Lambda_i)$.

  Conversely, let for instance $v_1\in \Vor(\Lambda_1)$, and let
  $w\in v_1+2\Lambda$. Let us write
  $w=v_1+2(u_1+u_2)= (v_1+2u_1)+2 u_2$ for $u_i \in \Lambda_i$. Since
  $v_1 \in \Vor(\Lambda_1)$, we have
  $\Vert v_1+2u_1 \Vert \geq \Vert v_1 \Vert$, with equality if and
  only if $v_1+2u_1=\pm v_1$. Thus,
\[
\Vert w \Vert^2 = \Vert v_1+2u_1 \Vert^2 +  \Vert 2u_2 \Vert^2 \geq
\Vert v_1+2u_1 \Vert^2 \geq \Vert v_1 \Vert^2
\]
with equality if and only if $u_2=0$ and $v_1+2u_1=\pm v_1$, namely
$w=\pm v_1$. So $v_1\in \Vor(\Lambda)$.
\end{proof}

\begin{theorem}\label{thm:OrthogonalSum}
Let $\Lambda$ be a lattice such that
\[
\Lambda = \Lambda_1 \perp \Lambda_2 \perp \ldots \perp \Lambda_m,
\quad \text{with } m \in \mathbb{N}.
\]
Then,
\[
\chi(\Lambda) = \max_{i\in \{ 1,\ldots,m \}} \chi(\Lambda_i).
\]
\end{theorem}

\begin{proof}
We again assume $\Lambda=\Lambda_1 \perp \Lambda_2$.

By Lemma~\ref{lem:VorVecSum},
$\Vor(\Lambda) = \Vor(\Lambda_1)\cup \Vor(\Lambda_2)$, and so
$\Cayley(\Lambda, \Vor(\Lambda_i))$ is a subgraph of
$\Cayley(\Lambda, \Vor(\Lambda))$. Hence,
$\chi(\Lambda)\geq\max\{\chi(\Lambda_1),\chi(\Lambda_2)\}$.

Conversely, let $k = \max \{ \chi(\Lambda_1), \chi(\Lambda_2) \}$. By
definition, for $i\in \{ 1,2 \}$, there is a proper coloring
$c_i : \Lambda_i \to \mathbb{Z}/k \mathbb{Z}$ such
that if $v_i - v'_i \in \Vor(\Lambda_i)$, then
$c_i(v_i)\neq c_i(v'_i) $. We shall show that
\[
\fonc{c}{\Lambda}{\mathbb{Z}/k\mathbb{Z}}{v_1+v_2}{c_1(v_1) + c_2(v_2)
  \bmod k}
\]
is a proper coloring of $\Lambda$. For this let $u, v \in \Lambda$ such
that $v = u+w$ with $w\in \Vor(\Lambda)$. Following
Lemma~\ref{lem:VorVecSum}, $w\in \Vor(\Lambda_1)\cup \Vor(\Lambda_2)$.
Assume for instance that $w=w_1 \in \Vor(\Lambda_1)$. Then, we write
$u = u_1 + u_2$ with $u_i \in \Lambda_i$, and
\[
c(v) = c_1(u_1+w_1) + c_2(u_2) \neq c_1(u_1) + c_2(u_2) = c(u) \bmod k.
\]
So $c$ is a proper coloring of $\Lambda$ and $\chi(\Lambda) \leq k$.
\end{proof}

\section{On the chromatic number of lattices of Voronoi's first kind}
\label{sec:firstkind}

In this section we give lower and upper bounds for the chromatic
number of lattices of Voronoi's first kind. Lattices of Voronoi's
first kind form a nice class of lattices: All lattices in dimensions
$2$ and $3$ belong to this class as well as $\mathsf{A}_n$ and
$\mathsf{A}_n^*$. Our lower and upper bounds coincide for all these
cases. For dimension $4$ and greater the bounds can differ. We like to pose
the question of computing the chromatic number of a lattice of
Voronoi's first kind as an open problem.

\subsection{Definitions and first examples}
\label{ssec:firstkinddefs}

Lattices of Voronoi's first kind are treated in detail for example
in~\cite{Conway1992a}. Here we start by collecting some facts about
them.

\begin{definition}
\label{def:firstkind}
A lattice $\Lambda$ is called a lattice \textit{of Voronoi's first
  kind} if it admits an \textit{obtuse superbasis}: There exist
lattice vectors $v_0,v_1,\ldots,v_n$ such that:
\begin{enumerate}
\item The set $\{ v_1,\ldots,v_n \}$ forms a basis of $\Lambda$,
\item We have $v_0+v_1+\cdots+v_n=0$,
\item For every $0\leq i < j \leq n$, the vectors $v_i$ and $v_j$
  enclose an obtuse angle, $v_i \cdot v_j \leq 0$.
\end{enumerate}
\end{definition}

The lattices $\mathsf{A}_n$ and $\mathsf{A}_n^*$ are of Voronoi's
first kind.  The lattice $\mathsf{A}_n$ possesses an obtuse
superbasis. Let $e_1,e_2, \ldots,e_{n+1}$ be the canonical basis of
$\mathbb{R}^{n+1}$. For $1\leq i \leq n$, let $v_i=e_i-e_{i+1}$, and
let $v_0=e_{n+1}-e_1$. Then $\{v_0,\ldots,v_n\}$ is an obtuse
superbasis of $\mathsf{A}_n$. The dual lattice $\mathsf{A}_n^*$ possesses a
\emph{strictly} obtuse superbasis. Let
\[
v_0=\left(-\frac{n}{n+1},\frac{1}{n+1},\ldots,\frac{1}{n+1}\right),\ldots,
v_n=\left(\frac{1}{n+1},\ldots,\frac{1}{n+1},-\frac{n}{n+1}\right) .
\]
Then $\{v_0,\ldots,v_n\}$ is an obtuse superbasis:
\[
v_i \cdot v_j = -1 \; \text{ for every } \;
0\leq i < j \leq n.
\]

If $\Lambda$ is a lattice of Voronoi's first kind, then it is known
(see \cite{Conway1992a}) that
\[
\Vor(\Lambda) \subseteq \left\{  v_I = \sum_{i\in I} v_i : I \subseteq \{0,
\ldots, n\}, I \neq \emptyset, I \neq \{0, \ldots, n\} \right\}.
\]

This immediately gives an upper bound for the chromatic number.

\begin{lemma}
\label{lem:VorFirstKindUpperBound}
Let $\Lambda$ be a lattice of Voronoi's first kind, with an obtuse
superbasis $\{v_0,\ldots,v_n\}$. Then, $\chi(\Lambda)$ is at most $n+1$.
\end{lemma}

\begin{proof}
By definition $\{ v_1,\ldots,v_n \}$ is a basis of $\Lambda$. Let us
show that the linear map
\[
\fonc{c}{\Lambda}{\mathbb{Z}/(n+1)\mathbb{Z}}{\sum\limits_{i=1}^n x_i v_i}{\sum\limits_{i=1}^n x_i \bmod (n+1)}
\]
is a proper coloring of $\Lambda$. Because of linearity it is enough
to check that it does not vanish on the strict Voronoi vectors. Let
$v_I$ be such a vector. If $0 \in I$ we replace $v_I$ by
$-v_I = v_{\{0,\ldots, n\} \setminus I}$ (by
Definition~\ref{def:firstkind} (2)) to make sure that $0 \notin I$.
Since $I$ is nontrivial, $0<|I|<n+1$. In other words, $c(v_I)\neq 0$.
\end{proof}

With this lemma it is easy to see that the chromatic numbers of
$\mathsf{A}_n$ and of its dual $\mathsf{A}_n^*$ are both equal to
$n + 1$. For $\mathsf{A}_n$ and for every $1\leq i < j \leq n$, the
vector $v_{\{i,i+1,\ldots,j\}}=e_i-e_{j+1}$ is a minimal vector, and
thus is a strict Voronoi vector. So,
\begin{equation}
\label{eq:anclique}
\{ 0, v_{\{1\}}, v_{\{1,2\}}, \ldots, v_{\{1,\ldots,n\}} \}
\end{equation}
is a clique in $\Cayley(\mathsf{A}_n, \Vor(\mathsf{A}_n))$ and so
$\chi(\mathsf{A}_n) \geq n+1$. For $\mathsf{A}_n^*$ we know (see
\cite{Conway1992a}) that every $v_I$ is a strict Voronoi
vector. Again, \eqref{eq:anclique} is a clique, and
$\chi(\mathsf{A}_n^*) \geq n+1$.

\subsection{Interpretation in terms of graphs and more general results}

In order to get a better understanding of the chromatic number of
lattices of Voronoi's first kind, we need to know which vectors $v_I$
are strict Voronoi vectors.

Let $\Lambda$ be a lattice of Voronoi's first kind with superbasis
$\{v_0,\ldots,v_n\}$.

\begin{definition}
The \emph{Delaunay graph} $D(\Lambda,\{v_0,\ldots,v_n\})$
is an undirected graph with vertex set $\{0,\ldots,n\}$ and where $i$
and $j$ are connected by an edge whenever $v_i \cdot v_j <0$.
\end{definition}

The combinatorics of the Delaunay graph
$D(\Lambda,\{v_0,\ldots,v_n\})$ determines the Cayley graph
$\Cayley(\Lambda, \Vor(\Lambda))$. Recall some standard terminology in
graph theory. Let $G = (V, E)$ be a graph, a subset of the vertex set
$U \subseteq V$ defines a \textit{cut} by
\[
\delta(U) = \{e \in E : |e \cap U| = 1\}.
\]
The strict Voronoi vectors of $\Lambda$ are the $v_I$ such that
the cut $\delta(I)$ is minimal with respect to inclusion, see for
example \cite{Dutour2009a} or \cite{Vallentin2003a}.
The minimal cuts are also known to be the ones such that,
when removing the edges of the cut, the number of connected
components in the graph increases by one.

A connected graph $G = (V, E)$ is called \textit{biconnected (a
  block)} if it remains connected when we remove any of its
vertices. One can decompose every connected graph into biconnected
components (block decomposition). The set of edges $E$ can be uniquely
written as a disjoint union $E = \bigcup_{i=1}^m E_i$ such that the
subgraph $G_i$ of $G$ induced by $E_i$ is a maximal biconnected
subgraph of $G$. Let $G_i$, with $i = 1, \ldots, m$, be the
biconnected components of the Delaunay graph
$D(\Lambda,\{v_0,\ldots,v_n\})$. Then, there exist lattices
$\Lambda_i$ of Voronoi's first kind with obtuse superbasis
$\mathcal{B}_i$ such that
\[
\Lambda = \Lambda_1 \perp \Lambda_2 \perp \ldots \perp \Lambda_m
\quad \text{and} \quad D(\Lambda_i,\mathcal{B}_i) = G_i
\]
holds, see for example \cite[Chapter 4, Chapter 5]{Oxley2011}.  Then
Theorem~\ref{thm:OrthogonalSum} and
Lemma~\ref{lem:VorFirstKindUpperBound} yield the following upper bound
for the chromatic number of $\Lambda$.

\begin{corollary}
  Let $\Lambda$ be a lattice of Voronoi's first kind with obtuse
  superbasis $\{v_0,\ldots,v_n\}$. Let $G_i$, $1\leq i \leq m$, be the
  biconnected components of the Delaunay graph
  $D(\Lambda,\{v_0,\ldots,v_n\}) $. Then,
\[
\chi(\Lambda) \leq \max_{i=1,\ldots,m} |V(G_i)|+1
\]
\end{corollary}

Now we go for lower bounds. Recall that a \textit{cycle} in a graph is
a collection of vertices $i_1, \ldots, i_l$ such that
$|\{i_1, \ldots, i_l\}| = l$ and such that $i_j$ is connected to
$i_{j+1}$, where indices are computed modulo $l$. Its length is equal
to $l$.

\begin{theorem}
  Let $\Lambda$ be a lattice of Voronoi's first kind with obtuse
  superbasis $\{v_0,\ldots,v_n\}$.  Then, the chromatic
  number of $\Lambda$ is at least the maximal length of a cycle in the
  Delaunay graph $D(\Lambda,\{v_0,\ldots,v_n\})$.
\end{theorem}

\begin{proof}
  Up to a permutation of the indices, we may assume that
  $\{0,1,\ldots,\sigma - 1 \}$ is a cycle $C$ in
  $D(\Lambda,\{v_0,\ldots,v_n\})$. We shall construct a clique of
  $\Cayley(\Lambda, \Vor(\Lambda))$ of size $\sigma$.

  For any $k$ in $\{0, \ldots, n\}$, we say that $c\in C$ is a
  \textit{connector} between $k$ and $C$ if there exists a path
  $\gamma = (k_1 = k, k_2, \ldots, k_s=c)$ in
  $D(\Lambda,\{v_0,\ldots,v_n\})$ such that $c$ is the only vertex on
  the path that belongs to $C$.

  Let $0 \leq \ell \leq \sigma - 1$. We define the set $I_\ell$ as the
  subset of all vertices $k$ in $\{0,\ldots, n \} $ such that every
  connector from $k$ to $C$ is in $\{0,1,\ldots,\ell \}$. In
  particular, for $\ell=\sigma - 1$,
  $I_{\sigma - 1}= \{0,\ldots, n \} $. An example is depicted in
  Figure \ref{fig:clique}. Then, if we define $u_\ell=v_{I_\ell}$ for
  all $0 \leq \ell \leq \sigma - 1$, then for every
  $0 \leq i < j \leq \sigma - 1$,
\[
 u_j-u_i=v_{I_{j}\setminus I_{i}}.
 \]
 In order to show that the set $\{ u_1, u_2,\ldots, u_\sigma \} $ is a
 clique in $\Cayley(\Lambda, \Vor(\Lambda))$, we need to prove that
 for every $0 \leq i < j \leq \sigma - 1$, the vector
 $v_{I_{j}\setminus I_{i}}$ is in $\Vor(\Lambda)$. Equivalently, since
 $D(\Lambda,\{v_0,\ldots,v_n\})$ is connected, we need to check that
 both $I_{j}\setminus I_{i}$ and its complementary in
 $\{0,\ldots, n \} $ induce connected subgraphs of
 $D(\Lambda,\{v_0,\ldots,v_n\})$.

 Take $k$ in $I_{j}\setminus I_{i}$. Since $k$ is not in $I_i$, there
 is a connector between $k$ and $C$ which is not in $\{0,\ldots,i \}$;
 but since $k$ is in $I_j$, this connector must be in
 $\{i+1,\ldots,j \}$. So $k$, as well as all the vertices in this
 path, are in $I_{j}\setminus I_{i}$ and are connected to
 $\{i+1,\ldots,j \}$, which is obviously connected and included in
 $I_{j}\setminus I_{i}$. Regarding the complementary set, take $k$ not
 in $I_{j}\setminus I_{i}$. Then either $k$ is not in $I_j$, and there
 is a connector between $k$ and $C$ in $\{j+1,\ldots, \sigma -1 \}$,
 or $k$ is in $I_i$, and every connector between $k$ and $C$ is in
 $\{0,\ldots,i \}$. Thus for every such $k$, one can find a path made
 of vertices not in $I_{j}\setminus I_{i}$, going to
 $\{j+1,\ldots, \sigma - 1 \} \cup \{0,\ldots,i \}$.

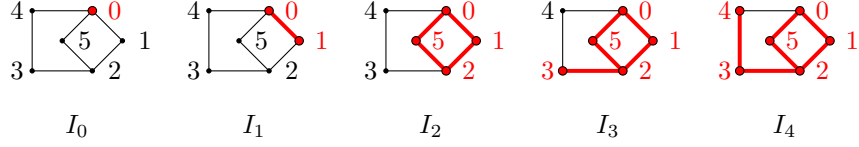
\begin{figure}
\begin{tabular}{ccccc}
\begin{tikzpicture}[scale=0.40]
\clip (-2,-2) rectangle (3,2);
\coordinate (p1) at (-1,1);
\coordinate (p2) at (1,1);
\coordinate (p3) at (1,-1);
\coordinate (p4) at (-1,-1);
\coordinate (p5) at (0,0);
\coordinate (p6) at (2,0);
\draw (p2) node [xshift=0.3cm,yshift=-0cm] {\textcolor{red}{$0$}} ;
\draw (p6) node [xshift=0.3cm,yshift=-0cm] {\textcolor{black}{$1$}} ;
\draw (p3) node [xshift=0.3cm,yshift=-0cm] {\textcolor{black}{$2$}} ;
\draw (p4) node [xshift=-0.2cm,yshift=0cm] {\textcolor{black}{$3$}} ;
\draw (p1) node [xshift=-0.2cm,yshift=0cm] {\textcolor{black}{$4$}} ;
\draw (p5) node [xshift=0.3cm,yshift=-0cm] {\textcolor{black}{$5$}} ;
\draw (p1) -- (p2) ;
\draw (p1) -- (p4) ;
\draw (p4) -- (p3) ;
\draw (p5) -- (p3) ;
\draw (p6) -- (p3) ;
\draw (p5) -- (p2) ;
\draw (p6) -- (p2) ;
\draw[fill=black] (p1) circle (2pt);
\draw[fill=red] (p2) circle (4pt);
\draw[fill=black] (p3) circle (2pt);
\draw[fill=black] (p4) circle (2pt);
\draw[fill=black] (p5) circle (2pt);
\draw[fill=black] (p6) circle (2pt);
\end{tikzpicture}
&
\begin{tikzpicture}[scale=0.40]
\clip (-2,-2) rectangle (3,2);
\coordinate (p1) at (-1,1);
\coordinate (p2) at (1,1);
\coordinate (p3) at (1,-1);
\coordinate (p4) at (-1,-1);
\coordinate (p5) at (0,0);
\coordinate (p6) at (2,0);
\draw (p2) node [xshift=0.3cm,yshift=-0cm] {\textcolor{red}{$0$}} ;
\draw (p6) node [xshift=0.3cm,yshift=-0cm] {\textcolor{red}{$1$}} ;
\draw (p3) node [xshift=0.3cm,yshift=-0cm] {\textcolor{black}{$2$}} ;
\draw (p4) node [xshift=-0.2cm,yshift=0cm] {\textcolor{black}{$3$}} ;
\draw (p1) node [xshift=-0.2cm,yshift=0cm] {\textcolor{black}{$4$}} ;
\draw (p5) node [xshift=0.3cm,yshift=-0cm] {\textcolor{black}{$5$}} ;
\draw (p1) -- (p2) ;
\draw (p1) -- (p4) ;
\draw (p4) -- (p3) ;
\draw (p5) -- (p3) ;
\draw (p6) -- (p3) ;
\draw (p5) -- (p2) ;
\draw[red, line width=1.5] (p6) -- (p2) ;
\draw[fill=black] (p1) circle (2pt);
\draw[fill=red] (p2) circle (4pt);
\draw[fill=black] (p3) circle (2pt);
\draw[fill=black] (p4) circle (2pt);
\draw[fill=black] (p5) circle (2pt);
\draw[fill=red] (p6) circle (4pt);
\end{tikzpicture}
&
\begin{tikzpicture}[scale=0.40]
\clip (-2,-2) rectangle (3,2);
\coordinate (p1) at (-1,1);
\coordinate (p2) at (1,1);
\coordinate (p3) at (1,-1);
\coordinate (p4) at (-1,-1);
\coordinate (p5) at (0,0);
\coordinate (p6) at (2,0);
\draw (p2) node [xshift=0.3cm,yshift=-0cm] {\textcolor{red}{$0$}} ;
\draw (p6) node [xshift=0.3cm,yshift=-0cm] {\textcolor{red}{$1$}} ;
\draw (p3) node [xshift=0.3cm,yshift=-0cm] {\textcolor{red}{$2$}} ;
\draw (p4) node [xshift=-0.2cm,yshift=0cm] {\textcolor{black}{$3$}} ;
\draw (p1) node [xshift=-0.2cm,yshift=0cm] {\textcolor{black}{$4$}} ;
\draw (p5) node [xshift=0.3cm,yshift=-0cm] {\textcolor{red}{$5$}} ;
\draw (p1) -- (p2) ;
\draw (p1) -- (p4) ;
\draw (p4) -- (p3) ;
\draw[red, line width=1.5] (p5) -- (p3) ;
\draw[red, line width=1.5] (p6) -- (p3) ;
\draw[red, line width=1.5] (p5) -- (p2) ;
\draw[red, line width=1.5] (p6) -- (p2) ;
\draw[fill=black] (p1) circle (2pt);
\draw[fill=red] (p2) circle (4pt);
\draw[fill=red] (p3) circle (4pt);
\draw[fill=black] (p4) circle (2pt);
\draw[fill=red] (p5) circle (4pt);
\draw[fill=red] (p6) circle (4pt);
\end{tikzpicture}
&
\begin{tikzpicture}[scale=0.40]
\clip (-2,-2) rectangle (3,2);
\coordinate (p1) at (-1,1);
\coordinate (p2) at (1,1);
\coordinate (p3) at (1,-1);
\coordinate (p4) at (-1,-1);
\coordinate (p5) at (0,0);
\coordinate (p6) at (2,0);
\draw (p2) node [xshift=0.3cm,yshift=-0cm] {\textcolor{red}{$0$}} ;
\draw (p6) node [xshift=0.3cm,yshift=-0cm] {\textcolor{red}{$1$}} ;
\draw (p3) node [xshift=0.3cm,yshift=-0cm] {\textcolor{red}{$2$}} ;
\draw (p4) node [xshift=-0.2cm,yshift=0cm] {\textcolor{red}{$3$}} ;
\draw (p1) node [xshift=-0.2cm,yshift=0cm] {\textcolor{black}{$4$}} ;
\draw (p5) node [xshift=0.3cm,yshift=-0cm] {\textcolor{red}{$5$}} ;
\draw (p1) -- (p2) ;
\draw (p1) -- (p4) ;
\draw[red, line width=1.5] (p4) -- (p3) ;
\draw[red, line width=1.5] (p5) -- (p3) ;
\draw[red, line width=1.5] (p6) -- (p3) ;
\draw[red, line width=1.5] (p5) -- (p2) ;
\draw[red, line width=1.5] (p6) -- (p2) ;
\draw[fill=black] (p1) circle (2pt);
\draw[fill=red] (p2) circle (4pt);
\draw[fill=red] (p3) circle (4pt);
\draw[fill=red] (p4) circle (4pt);
\draw[fill=red] (p5) circle (4pt);
\draw[fill=red] (p6) circle (4pt);
\end{tikzpicture}
&
\begin{tikzpicture}[scale=0.40]
\clip (-2,-2) rectangle (3,2);
\coordinate (p1) at (-1,1);
\coordinate (p2) at (1,1);
\coordinate (p3) at (1,-1);
\coordinate (p4) at (-1,-1);
\coordinate (p5) at (0,0);
\coordinate (p6) at (2,0);
\draw (p2) node [xshift=0.3cm,yshift=-0cm] {\textcolor{red}{$0$}} ;
\draw (p6) node [xshift=0.3cm,yshift=-0cm] {\textcolor{red}{$1$}} ;
\draw (p3) node [xshift=0.3cm,yshift=-0cm] {\textcolor{red}{$2$}} ;
\draw (p4) node [xshift=-0.2cm,yshift=0cm] {\textcolor{red}{$3$}} ;
\draw (p1) node [xshift=-0.2cm,yshift=0cm] {\textcolor{red}{$4$}} ;
\draw (p5) node [xshift=0.3cm,yshift=-0cm] {\textcolor{red}{$5$}} ;
\draw (p1) -- (p2) ;
\draw[red, line width=1.5] (p1) -- (p4) ;
\draw[red, line width=1.5] (p4) -- (p3) ;
\draw[red, line width=1.5] (p5) -- (p3) ;
\draw[red, line width=1.5] (p6) -- (p3) ;
\draw[red, line width=1.5] (p5) -- (p2) ;
\draw[red, line width=1.5] (p6) -- (p2) ;
\draw[fill=red] (p1) circle (4pt);
\draw[fill=red] (p2) circle (4pt);
\draw[fill=red] (p3) circle (4pt);
\draw[fill=red] (p4) circle (4pt);
\draw[fill=red] (p5) circle (4pt);
\draw[fill=red] (p6) circle (4pt);
\end{tikzpicture}
\\
$I_0$
&
$I_1$
&
$I_2$
&
$I_3$
&
$I_4$
\end{tabular}
\caption{Constructing a clique in $\Cayley(\Lambda, \Vor(\Lambda))$
  from the Delaunay graph $D(\Lambda,\{v_0,\ldots,v_n\}) $.}\label{fig:clique}
\end{figure}

\end{proof}

\begin{theorem}
\label{thm:An-and-An-star}
$\chi(\mathsf{A}_n) = \chi(\mathsf{A}_n^*) = n+1$.
\end{theorem}

\begin{proof}
For $\mathsf{A}_n^*$ the Delaunay graph
$D(\mathsf{A}_n^*, \{v_0,\ldots,v_n\})$ is the complete graph
$K_{n+1}$. For $\mathsf{A}_n$ the Delaunay graph
$D(\mathsf{A}_n, \{v_0,\ldots,v_n\})$ is a cycle of length $n+1$. In both
cases, our upper bound and lower bound coincide.
\end{proof}

\begin{example}
  In general, our upper bound differs from the lower bound.  Both
  bounds can be attained: On the left of Figure \ref{fig:Ex}, the longest cycle
  has size $4$ and one can find a coloring of the corresponding
  $5$-dimensional lattice with $4$ colors, whereas the lattice
  associated with the graph depicted on the right of Figure \ref{fig:Ex}, whose
  longest has size $5$, has chromatic number $6$. This can be seen by
  computing the chromatic number of a small subgraph of
  $\Cayley(\Lambda, \Vor(\Lambda))$.
\end{example}

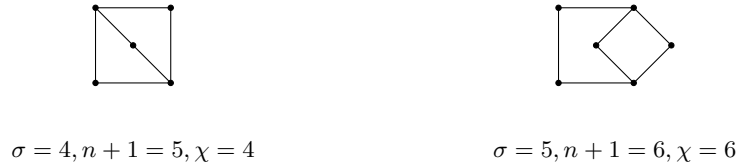
\begin{figure}[h!]
    \centering
    \begin{subfigure}[b]{0.5\textwidth}
        \centering
        \begin{tikzpicture}[scale=0.5]
\clip (-2,-2) rectangle (2,2);
\coordinate (p1) at (-1,1);
\coordinate (p2) at (1,1);
\coordinate (p3) at (1,-1);
\coordinate (p4) at (-1,-1);
\coordinate (p5) at (0,0);
\draw[fill=black] (p1) circle (2pt);
\draw[fill=black] (p2) circle (2pt);
\draw[fill=black] (p3) circle (2pt);
\draw[fill=black] (p4) circle (2pt);
\draw[fill=black] (p5) circle (2pt);
\draw (p1) -- (p2) ;
\draw (p1) -- (p3) ;
\draw (p1) -- (p4) ;
\draw (p2) -- (p3) ;
\draw (p3) -- (p4) ;
\end{tikzpicture}
       \caption*{$ \sigma=4, n+1=5, \chi=4 $}
    \end{subfigure}%
    ~
    \begin{subfigure}[b]{0.5\textwidth}
        \centering

        \begin{tikzpicture}[scale=0.5]
\clip (-2,-2) rectangle (3,2);
\coordinate (p1) at (-1,1);
\coordinate (p2) at (1,1);
\coordinate (p3) at (1,-1);
\coordinate (p4) at (-1,-1);
\coordinate (p5) at (0,0);
\coordinate (p6) at (2,0);
\draw[fill=black] (p1) circle (2pt);
\draw[fill=black] (p2) circle (2pt);
\draw[fill=black] (p3) circle (2pt);
\draw[fill=black] (p4) circle (2pt);
\draw[fill=black] (p5) circle (2pt);
\draw[fill=black] (p6) circle (2pt);
\draw (p1) -- (p2) ;
\draw (p1) -- (p4) ;
\draw (p4) -- (p3) ;
\draw (p5) -- (p3) ;
\draw (p6) -- (p3) ;
\draw (p5) -- (p2) ;
\draw (p6) -- (p2) ;
\end{tikzpicture}
    \caption*{$ \sigma=5, n+1=6, \chi=6 $}
    \end{subfigure}
    \caption{In the inequality $ \sigma \leq \chi \leq n+1 $, both bounds can be sharp.}\label{fig:Ex}
\end{figure}

\subsection{Application: the chromatic number of $3$-dimensional lattices}

Every lattice in dimensions $2$ and $3$ is of Voronoi's first kind,
see \cite{Conway1992a}. We can compute the chromatic number of these
lattices by applying our bounds which coincide in these cases, see
Table~\ref{tab:dim3}.

\begin{table}[!h]
\hspace*{-1.2cm}
\begin{tabular}{c|ccccc}
\multirow{2}{*}{\textit{lattice}}  & \multirow{2}{*}{$\mathbb{Z}^3$} & \multirow{2}{*}{$\mathsf{A}_2 \perp \mathbb{Z}$} & \multirow{2}{*}{$\mathsf{A}_3$} & \multirow{2}{*}{\begin{tiny}
$\mathbb{Z}\begin{pmatrix}
2\\0\\0
\end{pmatrix}\oplus \mathbb{Z}\begin{pmatrix}
0\\2\\0
\end{pmatrix} \oplus \mathbb{Z}\begin{pmatrix}
-1\\-1\\2
\end{pmatrix} $
\end{tiny}} & \multirow{2}{*}{$\mathsf{A}_3^*$}
\\
 &&&&&
\\[1ex]
  & \multirow{3}{*}{\begin{tikzpicture}[scale=0.35]
%[tdplot_main_coords,scale=1.1547]
\draw  (0,0,1) +(1,1,0) -- +(1,-1,0) -- +(-1,-1,0) -- +(-1,1,0) -- cycle ;
\draw [dashed] (0,0,-1) +(1,1,0) -- +(1,-1,0) -- +(-1,-1,0) -- +(-1,1,0) -- cycle ;
\draw   (0,0,-1)+(-1,1,0) --+(1,1,0) -- +(1,-1,0) ;
\draw (1,1,-1) -- (1,1,1);
\draw (1,-1,-1) -- (1,-1,1);
\draw (-1,1,-1) -- (-1,1,1);
\draw [dashed] (-1,-1,-1) -- (-1,-1,1);
\end{tikzpicture} }
&
\multirow{3}{*}{\tdplotsetmaincoords{70}{145}
\begin{tikzpicture}[scale=0.3,tdplot_main_coords]
\draw  (0,0,1) +(2,0,0) -- +(1,2*0.86602,0) -- +(-1,2*0.86602,0) -- +(-2,0,0) -- +(-1,-2*0.86602,0) -- +(1,-2*0.86602,0) -- cycle ;
\draw [dashed] (0,0,-1) +(2,0,0) -- +(1,2*0.86602,0) -- +(-1,2*0.86602,0) -- +(-2,0,0) -- +(-1,-2*0.86602,0) -- +(1,-2*0.86602,0) -- cycle ;
\draw   (0,0,-1)+(1,-2*0.86602,0) -- +(2,0,0)  -- +(1,2*0.86602,0) -- +(-1,2*0.86602,0)  ;
\draw (2,0,-1) -- (2,0,1);
\draw (1,2*0.86602,-1) -- (1,2*0.86602,1);
\draw (-1,2*0.86602,-1) -- (-1,2*0.86602,1);
\draw [dashed](-2,0,-1) -- (-2,0,1);
\draw [dashed] (-1,-2*0.86602,-1) -- (-1,-2*0.86602,1);
\draw  (1,-2*0.86602,-1) -- (1,-2*0.86602,1);
\end{tikzpicture}}
&
\multirow{3}{*}{\tdplotsetmaincoords{90}{65}
\begin{tikzpicture}[scale=0.3,tdplot_main_coords]
\draw [dashed] (-2,0,0)  -- (-1,-1,1) -- (0,-2,0)  -- (-1,-1,-1) -- cycle ;
\draw  (2,0,0)  -- (1,1,1) -- (0,2,0)  -- (1,1,-1) -- cycle ;
\draw  (2,0,0)  -- (1,-1,1) -- (0,-2,0)  -- (1,-1,-1) -- cycle ;
\draw [dashed](-2,0,0)  -- (-1,1,1) -- (0,2,0)  -- (-1,1,-1) -- cycle ;
\draw  (2,0,0)  -- (1,1,1) -- (0,0,2)  -- (1,-1,1) -- cycle ;
\draw [dashed] (-2,0,0)  -- (-1,-1,-1) -- (0,0,-2)  -- (-1,1,-1) -- cycle ;
\draw [dashed] (-2,0,0)  -- (-1,1,1) -- (0,0,2)  -- (-1,-1,1) -- cycle ;
\draw  (2,0,0)  -- (1,-1,-1) -- (0,0,-2)  -- (1,1,-1) -- cycle ;
\draw [dashed] (0,2,0)  -- (1,1,1) -- (0,0,2)  -- (-1,1,1) -- cycle ;
\draw  (0,-2,0)  -- (-1,-1,-1) -- (0,0,-2)  -- (1,-1,-1) -- cycle ;
\draw [dashed] (0,2,0)  -- (1,1,-1) -- (0,0,-2)  -- (-1,1,-1) -- cycle ;
\draw  (0,-2,0)  -- (-1,-1,1) -- (0,0,2)  -- (1,-1,1) -- cycle ;
 \end{tikzpicture}}
&
\multirow{3}{*}{\tdplotsetmaincoords{90}{35}
\begin{tikzpicture}[scale=0.4,tdplot_main_coords]
\draw [dashed] (-1,0,0) +(0,0,1) -- +(0,1,1/2) -- +(0,1,-1/2) -- +(0,0,-1) -- +(0,-1,-1/2) -- +(0,-1,1/2) --cycle ;
\draw [dashed] (0,1,0) +(0,0,1) -- +(1,0,1/2) -- +(1,0,-1/2) -- +(0,0,-1) -- +(-1,0,-1/2) -- +(-1,0,1/2) --cycle ;
\draw [dashed] (1/2,1/2,1) +(-1/2,-1/2,1/2) -- +(1/2,-1/2,0) -- +(1/2,1/2,-1/2) -- +(-1/2,1/2,0)  --cycle ;
\draw [dashed] (-1/2,1/2,-1) +(-1/2,1/2,1/2) -- +(1/2,1/2,0) -- +(1/2,-1/2,-1/2) -- +(-1/2,-1/2,0)  --cycle ;
\draw [dashed] (-1/2,1/2,1) +(1/2,-1/2,1/2) -- +(-1/2,-1/2,0) -- +(-1/2,1/2,-1/2) -- +(1/2,1/2,0)  --cycle ;
\draw  (1/2,-1/2,-1) +(1/2,-1/2,1/2) -- +(-1/2,-1/2,0) -- +(-1/2,1/2,-1/2) -- +(1/2,1/2,0)  --cycle ;
\draw [dashed] (1/2,1/2,-1) +(1/2,1/2,1/2) -- +(-1/2,1/2,0) -- +(-1/2,-1/2,-1/2) -- +(1/2,-1/2,0)  --cycle ;
\draw  (-1/2,-1/2,-1) +(-1/2,-1/2,1/2) -- +(1/2,-1/2,0) -- +(1/2,1/2,-1/2) -- +(-1/2,1/2,0)  --cycle ;
\draw  (1/2,-1/2,1) +(-1/2,1/2,1/2) -- +(1/2,1/2,0) -- +(1/2,-1/2,-1/2) -- +(-1/2,-1/2,0)  --cycle ;
\draw  (1,0,0) +(0,0,1) -- +(0,1,1/2) -- +(0,1,-1/2) -- +(0,0,-1) -- +(0,-1,-1/2) -- +(0,-1,1/2) --cycle ;
\draw  (0,-1,0) +(0,0,1) -- +(1,0,1/2) -- +(1,0,-1/2) -- +(0,0,-1) -- +(-1,0,-1/2) -- +(-1,0,1/2) --cycle ;
\draw  (-1/2,-1/2,1) +(1/2,1/2,1/2) -- +(-1/2,1/2,0) -- +(-1/2,-1/2,-1/2) -- +(1/2,-1/2,0)  --cycle ;
\end{tikzpicture}}
&
\multirow{3}{*}{\tdplotsetmaincoords{130}{60}
\begin{tikzpicture}[tdplot_main_coords,scale=0.3]
\draw  (2,0,-1)  -- +(0,1,1)  ;
\draw  (2,-1,0)  -- +(0,1,1)  ;
\draw [dashed] (-2,0,-1)  -- +(0,1,1)  ;
\draw [dashed] (-2,-1,0)  -- +(0,1,1)  ;
\draw  (0,1,-2)  -- +(0,1,1)  ;
\draw  (0,-2,1)  -- +(0,1,1)  ;
\draw  (2,0,1)  -- +(0,1,-1)  ;
\draw  (2,-1,0)  -- +(0,1,-1)  ;
\draw [dashed] (-2,0,1)  -- +(0,1,-1)  ;
\draw  (-2,-1,0)  -- +(0,1,-1)  ;
\draw [dashed] (0,1,2)  -- +(0,1,-1) ;
\draw  (0,-2,-1)  -- +(0,1,-1) ;
\draw  (0,2,-1)  -- +(1,0,1)  ;
\draw [dashed] (-1,2,0)  -- +(1,0,1)  ;
\draw  (0,-2,-1)  -- +(1,0,1)  ;
\draw  (-1,-2,0)  -- +(1,0,1)  ;
\draw  (1,0,-2)  -- +(1,0,1)  ;
\draw [dashed] (-2,0,1)  -- +(1,0,1)  ;
\draw [dashed] (0,2,1)  -- +(1,0,-1)  ;
\draw [dashed] (-1,2,0)  -- +(1,0,-1)  ;
\draw  (0,-2,1)  -- +(1,0,-1)  ;
\draw  (-1,-2,0)  -- +(1,0,-1)  ;
\draw  (1,0,2)  -- +(1,0,-1)  ;
\draw  (-2,0,-1)  -- +(1,0,-1)  ;
\draw  (0,-1,2)  -- +(1,1,0)  ;
\draw [dashed] (-1,0,2)  -- +(1,1,0)  ;
\draw  (0,-1,-2)  -- +(1,1,0) ;
\draw  (-1,0,-2)  -- +(1,1,0)  ;
\draw  (1,-2,0)  -- +(1,1,0)  ;
\draw [dashed] (-2,1,0)  -- +(1,1,0)  ;
\draw [dashed] (0,1,2)  -- +(1,-1,0)  ;
\draw [dashed] (-1,0,2)  -- +(1,-1,0)  ;
\draw  (0,1,-2)  -- +(1,-1,0) ;
\draw  (-1,0,-2)  -- +(1,-1,0)  ;
\draw  (1,2,0)  -- +(1,-1,0)  ;
\draw  (-2,-1,0)  -- +(1,-1,0)  ;
 \end{tikzpicture}}
\\
\textit{Voronoi}
&&&&&
\\
\textit{cell}&&&&&
\\
& \multirow{2}{*}{{\small cube}} & {\small hexagonal}  & {\small rhombic}  & {\small elongated}  & {\small truncated}
\\
&  &  {\small prism} &  {\small dodecahedron} &  {\small dodecahedron} &  {\small octahedron}
\\
&&&&&
\\[-3ex]
  & \multirow{5}{*}{\begin{tikzpicture}[scale=0.8]
\clip (-0,-0.5) rectangle (2,1.1);
\coordinate (p1) at (1/2,0.86602);
\coordinate (p2) at (1/2,0);
\coordinate (p3) at (2/2,0.86602);
\coordinate (p4) at (2/2,0);
\coordinate (p5) at (3/2,0.86602);
\coordinate (p6) at (3/2,0);
\draw[fill=black] (p1) circle (2pt);
\draw[fill=black] (p2) circle (2pt);
\draw[fill=black] (p3) circle (2pt);
\draw[fill=black] (p4) circle (2pt);
\draw[fill=black] (p5) circle (2pt);
\draw[fill=black] (p6) circle (2pt);
\draw (p1) -- (p2) ;
\draw (p4) -- (p3) ;
%\draw (p2) -- (p3) ;
%\draw (p2) -- (p3) ;
%\draw (p2) -- (p4) ;
\draw (p5) -- (p6) ;
\end{tikzpicture}}
&
\multirow{5}{*}{\begin{tikzpicture}[scale=0.8]
\clip (-0.1,-0.5) rectangle (2,1.1);
\coordinate (p1) at (1,0);
\coordinate (p2) at (0,0);
\coordinate (p3) at (1/2,0.86602);
\coordinate (p4) at (3/2,0.86602);
\coordinate (p5) at (3/2,0);
\draw[fill=black] (p1) circle (2pt);
\draw[fill=black] (p2) circle (2pt);
\draw[fill=black] (p3) circle (2pt);
\draw[fill=black] (p4) circle (2pt);
\draw[fill=black] (p5) circle (2pt);
\draw (p1) -- (p2) ;
\draw (p1) -- (p3) ;
\draw (p2) -- (p3) ;
%\draw (p2) -- (p3) ;
%\draw (p2) -- (p4) ;
\draw (p5) -- (p4) ;
\end{tikzpicture}}
&
\multirow{5}{*}{\begin{tikzpicture}[scale=0.5]
\clip (-1.5,-1.5) rectangle (1.3,1.1);
\coordinate (p1) at (-1,1);
\coordinate (p2) at (1,1);
\coordinate (p3) at (1,-1);
\coordinate (p4) at (-1,-1);
\draw[fill=black] (p1) circle (2pt);
\draw[fill=black] (p2) circle (2pt);
\draw[fill=black] (p3) circle (2pt);
\draw[fill=black] (p4) circle (2pt);
\draw (p1) -- (p2) ;
%\draw (p1) -- (p3) ;
\draw (p1) -- (p4) ;
\draw (p2) -- (p3) ;
%\draw (p2) -- (p4) ;
\draw (p3) -- (p4) ;
\end{tikzpicture}}
&
\multirow{5}{*}{\begin{tikzpicture}[scale=0.5]
\clip (-1.5,-1.5) rectangle (1.5,1.1);
\coordinate (p1) at (-1,1);
\coordinate (p2) at (1,1);
\coordinate (p3) at (1,-1);
\coordinate (p4) at (-1,-1);
\draw[fill=black] (p1) circle (2pt);
\draw[fill=black] (p2) circle (2pt);
\draw[fill=black] (p3) circle (2pt);
\draw[fill=black] (p4) circle (2pt);
\draw (p1) -- (p2) ;
\draw (p1) -- (p3) ;
\draw (p1) -- (p4) ;
\draw (p2) -- (p3) ;
%\draw (p2) -- (p4) ;
\draw (p3) -- (p4) ;
\end{tikzpicture}}
&
\multirow{5}{*}{\begin{tikzpicture}[scale=0.5]
\clip (-1.5,-1.5) rectangle (1.5,1.1);
\coordinate (p1) at (-1,1);
\coordinate (p2) at (1,1);
\coordinate (p3) at (1,-1);
\coordinate (p4) at (-1,-1);
\draw[fill=black] (p1) circle (2pt);
\draw[fill=black] (p2) circle (2pt);
\draw[fill=black] (p3) circle (2pt);
\draw[fill=black] (p4) circle (2pt);
\draw (p1) -- (p2) ;
\draw (p1) -- (p3) ;
\draw (p1) -- (p4) ;
\draw (p2) -- (p3) ;
\draw (p2) -- (p4) ;
\draw (p3) -- (p4) ;
\end{tikzpicture}}
\\
\textit{Delaunay}&&&&&
\\
\textit{graph}&&&&&
\\
&&&&&
\\
\textit{chromatic} & \multirow{2}{*}{$2$} &\multirow{2}{*}{$3$} & \multirow{2}{*}{$4$} & \multirow{2}{*}{$4$} & \multirow{2}{*}{$4$}
\\
\textit{number} &&&&&
\\
\end{tabular}
\caption{The chromatic number of $3$-dimensional lattices}\label{tab:dim3}
\end{table}

\section{Sphere packing lower bounds}
\label{sec:spherepackingbound}

In this section we prove lower bounds for the chromatic number of a
lattice by considering connections to the classical sphere packing
problem.

A subset $\mathcal{P}$ of $\mathbb{R}^n$ defines a packing of unit
spheres if the distance between all pairs of distinct points in
$\mathcal{P}$ is at least $2$. Define the center density of
$\mathcal{P}$ as the number of points in $\mathcal{P}$ per unit
volume, more precisely the \emph{(upper) center density} of
$\mathcal{P}$ is defined as
\[
\delta(\mathcal{P}) = \limsup_{R \to \infty} \frac{|\mathcal{P} \cap
  [-R,R]^{n}|}{\vol [-R,R]^{n}},
\]
where $[-R,R]^n$ is the regular $n$-dimensional cube with side length
$2R$. The largest center density of a packing of unit spheres in
$\mathbb{R}^n$ is
\[
\delta_{\mathbb{R}^n} = \sup\{\delta(\mathcal{P}) : \mathcal{P}
\subseteq \mathbb{R}^n \text{ defines a packing of unit spheres}\}.
\]

\begin{theorem}
\label{thm:spherepackingbound}
Let $\Lambda$ be an $n$-dimensional lattice which defines a packing of
unit spheres. Let $\rho$ be a positive real number so that all strict
Voronoi vectors of $\Lambda$ have length strictly less than
$\rho$. Then,
\[
\chi(\Lambda) \geq \left(\frac{\rho}{2}\right)^n \frac{\delta(\Lambda)}{\delta_{\mathbb{R}^n}}.
\]
\end{theorem}

\begin{proof}
Suppose that the chromatic number of $\Lambda$ equals $k$. Then one
can decompose $\Lambda$ into $k$ color classes $C_1, \ldots, C_k$. We
may assume that the first color class $C_1$ has the largest density
among these color classes. In particular, inequality
\[
k \delta(C_1) \geq \delta(\Lambda)
\]
holds. Then for all $v, w \in C_1$ with $v \neq w$ we have
$\|v - w\| \geq \rho$. Hence, $\frac{2}{\rho} C_1$ defines a packing of unit
spheres. So,
\[
\delta_{\mathbb{R}^n} \geq \delta\left(\frac{2}{\rho} C_1\right)  = \left(\frac{\rho}{2}\right)^n \delta(C_1),
\]
and the claim of the theorem follows by combining the two inequalities above.
\end{proof}

The following lemma gives a lower bound for $\rho$.

\begin{lemma}
\label{lem:length}
Let $\Lambda$ be an $n$-dimensional lattice which defines a packing of
unit spheres. If $v \in \Lambda \setminus
\{0\}$ is not a strict Voronoi vector, then $\|v\| \geq \sqrt{8}$.
\end{lemma}

\begin{proof}
  By Voronoi's characterization of strict Voronoi vectors
  \eqref{eq:Voronoi-characterization} there is a lattice vector
  $w \in v + 2\Lambda$ with $w \neq \pm v$ and $\|w\| \leq \|v\|$. We
  may assume that $v \cdot w \geq 0$, otherwise we replace $w$ by its
  negative $-w$. Define
  $u = \frac{1}{2}(v - w) \in \Lambda \setminus \{0\}$. Then,
\[
4\|u\|^2 = \|v - w\|^2 = \|v\|^2 - 2 v \cdot w + \|w\|^2 \leq \|v\|^2
+ \|w\|^2 \leq 2\|v\|^2.
\]
So, $\|v\| \geq \sqrt{2} \|u\| \geq \sqrt{8}$, since $\|u\| \geq 2$.
\end{proof}

\subsection{First application: Chromatic number of $\mathsf{E}_8$ and of the Leech lattice}
\label{ssec:SpherePackingLeechE8}

\begin{theorem}\label{thm:Leech}
The chromatic number of the Leech lattice $\Lambda_{24}$ equals
$4096$.
\end{theorem}

\begin{proof}
  It is known that the strict Voronoi vectors of $\Lambda_{24}$ are
  all vectors $v \in \Lambda_{24}$ with $v \cdot v \in \{4,6\}$, see
  \cite{Conway1988a}. Dong, Li, Mason, and Norton showed \cite[Theorem
  4.1]{Dong1998a} that one can find an isometric
  copy of $\sqrt{2}\Lambda_{24}$ as a sublattice $\Gamma$ of
  $\Lambda_{24}$; in \cite{Nebe2014a} Nebe and Parker classified all $16$
  orbits of such sublattices under the action of the automorphism
  group of $\Lambda_{24}$. Clearly, such a sublattice $\Gamma$ has index
  $2^{12} = 4096$ and any nonzero vector $v$ in this sublattice
  satisfies $v \cdot v \geq 8$.  Hence, we can color $\Lambda_{24}$ by
  using the $4096$ cosets $v + \Gamma$ as
  color classes, with $v \in \Lambda_{24}$. Thus,  $\chi(\Lambda_{24}) \leq 4096$.

  For the lower bound, we apply
  Theorem~\ref{thm:spherepackingbound}. The Leech lattice defines the
  densest sphere packing in dimension $24$, see \cite{Cohn2017a},
  $\delta(\Lambda_{24}) = \delta_{\mathbb{R}^{24}}$.  We can apply
  Theorem~\ref{thm:spherepackingbound} with $\rho = \sqrt{8}$ and
  get
\[
\chi(\Lambda_{24}) \geq \left(\frac{\sqrt{8}}{2}\right)^{24} = 4096.
\]
\end{proof}

Similarly, one can show $\chi(\mathsf{E}_{8}) = 16$ by using the
fact that $\mathsf{E}_8$ is the densest sphere packing in dimension
$8$, see \cite{Viazovska2017a}.

\subsection{Second application: Exponential growth of the chromatic
  number}
\label{ssec:asymptotic}

In this section we investigate the asymptotic behaviour of the chromatic
number of lattices when the dimension tends to infinity.

For a lattice $\Lambda \subset \mathbb{R}^n$, we denote by
$\mu = \mu(\Lambda)$ the norm of its smallest nonzero vector. Then
$\frac{2}{\mu} \Lambda$ defines a packing of unit spheres, and
we extend the notation $\delta(\Lambda)$ for the center density of this
packing by
\[
\delta(\Lambda) := \delta\left(\frac{2}{\mu} \Lambda \right) = \frac{1}{\vol \left(\mathbb{R}^n/(\frac{2}{\mu}  \Lambda)\right)}
= \frac{\mu^n}{2^n \cdot \vol(\mathbb{R}^n/ \Lambda)}.
\]

The best upper bound known for $\delta_{\mathbb{R}^n}$ is the
Kabatiansky-Levenshtein bound, \cite{Kabatiansky1978a}, which states
\[
\delta_{\mathbb{R}^n} \leq 2^{(-0.5990\ldots + o(1))n} \cdot V_n^{-1},
\]
where $V_n$ is the volume of the $n$-dimensional unit ball. So by using
Theorem~\ref{thm:spherepackingbound} and Lemma~\ref{lem:length}, we get,
for any lattice $\Lambda \subset \mathbb{R}^n$,
\begin{equation}\label{eq:ChiExponential}
\chi(\Lambda) \geq (\sqrt{2})^n \cdot 2^{(0.5990\ldots - o(1))n} \cdot V_n \cdot
 \delta(\Lambda) =  2^{(1.0990\ldots - o(1))n} \cdot V_n \cdot \delta(\Lambda).
\end{equation}

Let us now recall Siegel's mean value theorem (see
\cite{Siegel1945a}): For any lattice $\Lambda \subset \mathbb{R}^n$
and any $r>0$, we denote by $N_\Lambda(r)$ the number of nonzero
lattice vectors of $\Lambda$ in the open ball $B(r)$ having radius
$r$. The Siegel mean value theorem states that the expected value of
$N_\Lambda(r)$ in a random lattice $\Lambda$ of volume $1$ is
\[
\mathbb{E}[N_\Lambda(r)]=\vol(B(r)).
\]
To compute the expectation one uses the Haar measure on
$\mathrm{SL}_n(\mathbb{R})/\mathrm{SL}_n(\mathbb{Z})$.

This equality implies two remarkable consequences: First, by choosing
$r_n$ such that $\vol(B(r_n)) = 2$, it proves the existence of a lattice
$\Lambda_n$ with strictly less than two nonzero vectors in $B(r_n)$.
Since such a vector would come with its opposite, the minimum of
$\Lambda_n$ has to be at least $r_n$, and therefore the density of
$\Lambda_n$ satisfies
\[
\delta(\Lambda_n) \geq \left(\frac{r_n}{2}\right)^n =
\frac{\vol(B(r_n))}{2^n \cdot V_n} = \frac{2}{2^n\cdot V_n},
\]
which essentially is the Minkowski-Hlawka lower bound for
lattice sphere packings (see \cite{Hlawka1943}). The second
consequence concerns the density of a random lattice: Let us fix
$\varepsilon > 0$, and $r_n$ such that
$\vol(B(r_n))= 2 \cdot (1 + \varepsilon)^{-n}$.  Following the
previous idea, whenever $N_{\Lambda_n}(r_n)<2$, the density of
$\Lambda_n$ satisfies
\[
\delta(\Lambda_n) \geq \frac{2}{(2(1 + \varepsilon))^n \cdot V_n}.
\]
By using Siegel's mean value theorem
and Markov's inequality, we prove that this happens with high
probability when $n$ grows:

\[
\mathbb{P}[N_{\Lambda_n}(r_n) \geq 2]\leq
\frac{\mathbb{E}[N_{\Lambda_n}(r_n)]}{2} = \frac{1}{(1 +
  \varepsilon)^n} \to  0,
\]
if $n$ tends to infinity.

Combined with \eqref{eq:ChiExponential}, these observations provide:

\begin{theorem}
  With high probability, the chromatic number of a random
  $n$-dim\-ensional lattice grows exponentially in $n$.  Moreover, there
  are $n$-dimensional lattices $\Lambda_n$ with
\[
\chi(\Lambda_n) \geq 2 \cdot 2^{(0.0990\ldots - o(1))n}.
\]
\end{theorem}

Note that the Minkowski-Hlawka lower bound has been improved, in such
a way that the constant $2$ in the numerator could be replaced with
some quasi-linear function in $n$, see \cite{Rogers1947a},
\cite{Ball1992a}, \cite{Vance2011a}, \cite{Venkatesh2013a}. Even
though any random lattice should be dense, and consequently should
have an exponential chromatic number, to date there is no efficient
way to construct such a lattice: The only algorithms for this purpose
run in exponential time with respect to the dimension (see
\cite{Moustrou2017a}).

However, explicit examples of subexponential growth are known.  For
$n\geq 3$ the cut polytope $\CUT_n$ (see \cite{Deza1997a}) is an
$n(n-1)/2$-dimensional polytope which is a Delaunay polytope of
lattice (see \cite{DezaGrishukhin1996}) . The vertex-edge graph of
$\CUT_n$ is the complete graph on $2^{n-1}$ vertices.  Thus we get an
infinite family of Delaunay polytope with chromatic number lower
bounded by $2^{O(\sqrt{n})}$.

We think it is an interesting question to construct explicit families
of lattices whose chromatic number grows exponentially with the
dimension.

\section{Spectral lower bounds}
\label{sec:spectralbound}

In this section we derive a lower bound for the chromatic number of a
lattice where we apply the generalization of Hoffman's bound as
developed by Bachoc, DeCorte, Oliveira, and Vallentin
\cite{Bachoc2014a}. In Section~\ref{ssec:setup} we recall some
terminology and facts from \cite{Bachoc2014a}. Then we compute the
spectral bound for the irreducible root lattices case by case in
Section~\ref{ssec:spectralbound-root}. Table~\ref{table:SpectralIrreducible}
summarizes the results obtained.

\subsection{Setup}
\label{ssec:setup}

For an $n$-dimensional lattice $\Lambda \subseteq \mathbb{R}^n$ define
the (complex) Hilbert space
\[
\ell^2(\Lambda) = \left\{ f \colon \Lambda \to \mathbb{C} : \sum_{u
    \in \Lambda}  |f(u)|^2 < \infty \right\}
\]
which has inner product
\[
\langle f, g \rangle = \sum_{u \in \Lambda} f(u) \overline{g(u)}.
\]
The \emph{convolution} of two elements $f, g \in \ell^2(\Lambda)$ is
$f * g \in \ell^2(\Lambda)$ defined by
\[
(f * g)(v) = \sum_{u \in \Lambda} f(u) g(v-u).
\]

Assume that $\mu \in \ell^2(\Lambda)$ is real-valued, that its support
is contained in $\Vor(\Lambda)$, and that it satisfies
$\mu(v) = \mu(-v)$ for all $v \in \Lambda$. Consider the convolution
operator
\[
A_{\mu} : \ell^2(\Lambda) \to \ell^2(\Lambda)
\]
defined by
\[
A_{\mu} f (v) = \sum_{u \in \Vor(\Lambda)} \mu(u) f(v-u) = (\mu * f)(v).
\]
In a certain sense, $A_{\mu}$ is a weighted adjacency operator of
~$\Cayley(\Lambda, \Vor(\Lambda))$. It is easy to verify that
$A_{\mu}$ is a bounded and self-adjoint operator. Its \emph{numerical
  range} is
\[
W(A_\mu) = \{\langle A_\mu f, f \rangle : f \in \ell^2(\Lambda), \langle
f, f \rangle = 1\}.
\]
The numerical range is known to be an interval in~$\mathbb{R}$. By
\[
\begin{split}
& m(A_{\mu}) = \inf\{\langle A_\mu f, f \rangle : f \in \ell^2(\Lambda),
\langle f, f \rangle = 1\},\\
& M(A_{\mu}) = \sup\{\langle A_\mu f, f \rangle : f \in \ell^2(\Lambda),
\langle f, f \rangle = 1\}
\end{split}
\]
we denote the endpoints of the interval $W(A_{\mu})$.

We say that a subset $I \subseteq \Lambda$ is an \emph{independent
  set} of the operator $A_{\mu}$ if $\langle A_{\mu} f, f \rangle = 0$
for each $f \in \ell^2(\Lambda)$ which vanishes outside of $I$. The
\emph{chromatic number} of~$A_{\mu}$ is the smallest number $k$ so
that one can partition $\Lambda$ into $k$ independent sets. By
\cite[Theorem 2.3]{Bachoc2014a} one has the following lower bound for
$\chi(A_\mu)$:
\[
1 - \frac{M(A_\mu)}{m(A_\mu)} \leq \chi(A_\mu).
\]
Since every independent set of $\Cayley(\Lambda, \Vor(\Lambda))$ is
also an independent set of the operator $A_\mu$ we see that
\[
\chi(A_\mu) \leq \chi(\Cayley(\Lambda, \Vor(\Lambda))) = \chi(\Lambda).
\]

Therefore, we are interested in determining the parameters $m(A_\mu)$
and $M(A_\mu)$ and in choosing $\mu$ so that the bound becomes as
large as possible.

For determining $m(A_\mu)$ and $M(A_\mu)$ for a given convolution
operator $A_\mu$ we apply standard facts from Fourier analysis, in
particular the Parseval identity, the theorem of Riesz-Fischer, and
the fact that the Fourier transform of a convolution is a product.

The only non-standard item here is that in standard texts on Fourier
analysis, see for example \cite{Dym1972a}, the role of primal and dual
spaces are interchanged. In order not to confuse the reader (and to
some extend not to confuse the authors) we consider a new
$n$-dimensional lattice $\Gamma$. In our context $\Gamma$ will play
the role of the dual lattice $\Lambda^*$.  When $\Gamma = \Lambda^*$,
then $\Gamma^* = (\Lambda^*)^* = \Lambda$.

Consider the Hilbert space of square-integrable $\Gamma$-periodic
functions
\[
L^2(\mathbb{R}^n / \Gamma) = \left\{F \colon \mathbb{R}^n / \Gamma
  \to \mathbb{C} : \int_{\mathbb{R}^n / \Gamma} |F(x)|^2 \, dx < \infty \right\}
\]
with inner product
\[
(F,G) = \int_{\mathbb{R}^n / \Gamma} F(x) \overline{G(x)}\, dx,
\]
where we normalize the Lebesgue measure $dx$ so that the
$n$-dimensional volume of a fundamental domain
$\vol(\mathbb{R}^n / \Gamma) $ equals one. The exponential functions
\[
E_v : x \mapsto e^{2\pi i v \cdot x} \quad \text{with} \quad v \in \Gamma^*
\]
form a complete orthonormal system for $L^2(\mathbb{R}^n /
\Gamma)$. We define the \emph{Fourier coefficient} of $F$ at $v$ by
\[
\widehat{F}(v) = (F, E_v) \quad \text{with} \quad v \in \Gamma^*.
\]
By Parseval's identity and by the Riesz-Fischer theorem the map
\[
\widehat{}\, \colon L^2(\mathbb{R}^n / \Gamma) \to \ell^2(\Gamma^*),
\quad \widehat{F}(v) = (F, E_v)
\]
is an isometry:
\[
(F,G) = \langle \widehat{F}, \widehat{G} \rangle \quad \text{for all }
F, G \in L^2(\mathbb{R}^n / \Gamma).
\]
We consider two functions $f, g \in \ell^2(\Gamma^*)$. By the isometry
of~$~\widehat{}~$ there are functions $F, G \in L^2(\mathbb{R}^n / \Gamma)$
with
\[
\widehat F = f \quad \text{and} \quad \widehat G = g.
\]
Furthermore,
\[
(f * g)(v) = \widehat{F \cdot G}(v).
\]

\smallskip

Back to the lattice $\Lambda = \Gamma^*$ and the convolution operator
$A_\mu$. For $\mu = \widehat{G}$,
$f = \widehat{F} \in \ell^2(\Lambda)$, we have
\[
G(x) = \sum_{u \in \Vor(\Lambda)} \mu(u) e^{2\pi i u
    \cdot x},
\]
and
\[
\begin{split}
\langle A_\mu f, f \rangle = \; & \langle \mu * f, f \rangle = \langle
\widehat{G \cdot F}, \widehat{F} \rangle = (GF,F)\\
= \; &
\int_{\mathbb{R}^n /
  \Lambda^*} \sum_{u \in \Vor(\Lambda)} \mu(u) e^{2\pi i u \cdot x}
|F(x)|^2 \, dx.
\end{split}
\]
Hence, by choosing two appropriate sequences (see \cite[Section
3.1]{Bachoc2014a})) approximating the corresponding Dirac measures, we see
\[
\begin{split}
m(A_\mu) = \inf\left\{ \sum_{u \in \Vor(\Lambda)} \mu(u) e^{2\pi i u
    \cdot x} : x \in \mathbb{R}^n / \Lambda^*\right\},\\
M(A_\mu) = \sup\left\{ \sum_{u \in \Vor(\Lambda)} \mu(u) e^{2\pi i u
    \cdot x} : x \in \mathbb{R}^n / \Lambda^*\right\}.\\
\end{split}
\]

We summarize our considerations in the following theorem.

\begin{theorem}
  Let $\Lambda \subseteq \mathbb{R}^n$ be an $n$-dimensional lattice.
  Suppose that $\mu \in \ell^2(\Lambda)$ is real-valued,
  $\mu(v) = \mu(-v)$ for all $v \in \Lambda$, and the support of $\mu$
  is contained in $\Vor(\Lambda)$. Then,
\[
\chi(\Lambda) \geq 1 - \frac{\sup\limits_{x \in \mathbb{R}^n / \Lambda^*}
\sum\limits_{u \in \Vor(\Lambda)} \mu(u) e^{2\pi i u \cdot x}}{\inf\limits_{x \in \mathbb{R}^n / \Lambda^*} \sum\limits_{u \in \Vor(\Lambda)} \mu(u) e^{2\pi i u
    \cdot x}}.
\]
\end{theorem}

If we uniformly choose
\[
\mu(v) =
\begin{cases}
1/|\Vor(\Lambda)|, &  \text{for } v \in \Vor(\Lambda),\\
0, & \text{otherwise,}
\end{cases}
\]
then the bound in the previous theorem simplifies to the following
generalization of the Hoffman bound for infinite graphs:

\begin{corollary}\label{cor:Hoffman}
Let $\Lambda \subseteq \mathbb{R}^n$ be an $n$-dimensional
lattice. Then,
\[
\chi(\Lambda) \geq 1 - \left(\inf\limits_{x \in \mathbb{R}^n / \Lambda^*} \frac{1}{|\Vor(\Lambda)|} \sum_{u \in \Vor(\Lambda)}
  e^{2\pi i u \cdot x}\right)^{-1}.
\]
\end{corollary}

\subsection{Computing the spectral bound for irreducible root lattices}
\label{ssec:spectralbound-root}

In this section we compute the lower bound given by
Corollary~\ref{cor:Hoffman} for each of the irreducible root lattices.
By the classification of Witt these are the families of lattices
$\mathsf{A}_n$, $\mathsf{D}_n$ and the three sporadic lattices
$\mathsf{E}_6$, $\mathsf{E}_7$, and $\mathsf{E}_8$.

As recalled in Section~\ref{ssec:rootlattices}, the set
$\Vor(\Lambda)$ of strict Voronoi vectors for a root lattice $\Lambda$
is simply its set of roots.  Now the latter constitutes a \emph{root
  system} (sometimes known as a ``crystallographic'' root system, but
these are the only ones which we will consider):
see~\ref{par:rootsystems}; the root systems are themselves classified,
\cite[\S11]{Humphreys1972}, and the ones which arise from root
lattices are known as ``simply laced'' or ``A-D-E'' root systems.  (We
will review the simply laced irreducible root systems below along with
the corresponding lattices.)

Computing the lower bound of Corollary~\ref{cor:Hoffman} for an
irreducible root lattice is therefore equivalent to finding the
smallest value attained by the Fourier transform of a simply laced
irreducible root system. Here, a finite set
$\Phi \subseteq \mathbb{R}^n$ defines the function
$\mathscr{F}_\Phi\colon x \mapsto \sum_{u \in \Phi} e^{2\pi i u \cdot
  x}$
on $\mathbb{R}^n$ which is the Fourier transform of the sum of delta
measures concentrated on the elements of~$\Phi$. Let us emphasize that
since $\Phi$ is a (crystallographic!) root system, this function is,
in fact, a trigonometric polynomial; and since $\Phi$ is symmetric
($\Phi = -\Phi$), it is real.

Now this reformulation affords a link with representation theory (the
required facts of which are recalled in Appendix~\ref{sec:liegroups}).
Namely, if $\Phi$ is a (reduced, but not necessarily simply laced)
root system of rank~$n$ then the function $n + \mathscr{F}_\Phi$ is
``essentially'' the character of the adjoint representation of the ---
say, simply connected --- compact real Lie group $G_\Phi$ associated
with~$\Phi$ (namely $\mathrm{SL}_{n+1}$ in the case of $\mathsf{A}_n$,
or $\mathrm{Spin}_{2n}$ in the case of $\mathsf{D}_n$, or one of the
exceptional Lie groups $\mathsf{E}_6,\mathsf{E}_7,\mathsf{E}_8$ in the
case of the correspondingly named
$\mathsf{E}_6,\mathsf{E}_7,\mathsf{E}_8$); the precise statement and
explanation of why the two problems are equivalent is given in
Proposition~\ref{prop:adjcharvals}.

The problem of computing our spectral lower bound is therefore
essentially that of computing the least value attained by the adjoint
character of a simple compact real Lie group of type~A-D-E.  The
values in question have been considered and computed by Serre in
\cite[Theorem 3']{Serre2004a}: In Table
\ref{table:SpectralIrreducible} we state Serre's result in the form in
which it is useful for the main part of the paper.

\begin{table}[h]
\begin{tabular}{ccl}
\textit{lattice} & \textit{spectral lower bound}\\[0.5ex]
\hline
\\[-2ex]
$\mathsf{A}_n$ & $n+1$ & Thm. \ref{thm:lowera}\\
$\mathsf{D}_n$ & $n$, when $n$ even & Thm. \ref{thm:lowerd}\\
& $n+1$, when $n$ odd & Thm. \ref{thm:lowerd}\\
$\mathsf{E}_6$ & $9$ & Thm. \ref{Theorem_CritValues_E6}, Sec. \ref{ssec:e6Schlafi}\\
$\mathsf{E}_7$ & $10$ & Thm. \ref{Theorem_CritValues_E7}, Sec. \ref{ssec:e7e8sos}\\
$\mathsf{E}_8$ & $16$ & Thm. \ref{Theorem_CritValues_E8}, Sec. \ref{ssec:e7e8sos}\\[1ex]
\end{tabular}
\caption{The spectral lower bounds on the chromatic number for the
  irreducible root lattices.}
\label{table:SpectralIrreducible}
\end{table}

Again, the equivalence of the result as stated here with that as
stated in Serre's note is provided by
Proposition~\ref{prop:adjcharvals}.

Serre's note does not contain a proof of the stated
result\footnote{Serre writes: ``The classical groups are easy enough,
  but $\mathsf{F}_4$, $\mathsf{E}_6$, $\mathsf{E}_7$ and
  $\mathsf{E}_8$ are not (especially $\mathsf{E}_6$, which I owe to
  Alain Connes). I hope there is a better proof.''}. In this paper we
provide one for the A-D-E case. We treat $\mathsf{A}_n$ and
$\mathsf{D}_n$ in Theorem \ref{thm:lowera} and
Theorem~\ref{thm:lowerd} below, and we defer the case $\mathsf{E}_n$
to Appendix~\ref{sec:serresproof}. For $\mathsf{E}_7$ and
$\mathsf{E}_8$ we suggest an alternative proof technique based on sums
of squares and semidefinite optimization in
Section~\ref{ssec:e7e8sos}. Using the link to the chromatic number of
lattices, the case $\mathsf{E}_6$ turns out to be the easiest of the
$\mathsf{E}_n$ cases. We treat it in Section~\ref{ssec:e6Schlafi}.

\subsubsection{$\mathsf{A}_n$}

The irreducible root lattice $\mathsf{A}_n$ has $n(n+1)$ roots
\[
R(\mathsf{A}_n) = \{e_r - e_s : 0 \leq r,s \leq n, r \neq s\},
\]
where $e_r$ denotes the $r$-th standard unit vector in $\mathbb{R}^{n+1}$.

\begin{theorem}\label{thm:lowera}
The critical values of ${\mathcal F}_{\mathsf{A}_n}$ are $n(n+1)$ and $-(n+1)$.
In particular, $\chi(\mathsf{A}_n) \geq n+1$.
\end{theorem}

\begin{proof}
Given $x \in \mathbb{R}^{n+1}$ such that $x_0 + \cdots + x_n = 0$, we
define $z_0,\ldots,z_n$ by $z_r = e^{2\pi i x_r}$, so that $z_0\cdots z_n = 1$.
The sum
\[
S = {\mathcal F}_{\mathsf{A}_n}(x) = \sum_{u \in R(\mathsf{A}_n)}
e^{2\pi i u \cdot x}
\]
is equal to
\begin{equation*}
\begin{array}{rcl}
  S
&=& \sum\limits_{r\neq s} z_r/z_s\\
&=& \sum\limits_{r,s} z_r/z_s - (n+1)\\
&=& \left(\sum\limits_r z_r\right) \left(\sum\limits_s 1/z_s\right) - (n+1)\\
&=& \left(\sum\limits_r z_r\right) \overline{\left(\sum\limits_s z_s\right)} - (n+1)\\
&=& \left\vert \sum\limits_r z_r\right\vert^2 - (n+1).
\end{array}
\end{equation*}
Clearly $\left\vert \sum_r z_r\right\vert^2$ critical values are $0$
(when $\sum_i z_i=0$) and $(n+1)^2$ (when all $z_i$ are
identical). The critical values can be obtained when $z_0\cdots z_n =
1$, for example by the $(n+1)$-th roots of unity\footnote{
For example we have $z_r = e^{2 \pi i/(n+1)}$ for all $r$, by
letting $x_1 = \cdots = x_n = 1/(n+1)$ and $x_0 = 1/(n+1) - 1 =
-n/(n+1)$. More explicitly, we have a bijection 
between the set of all $x \in \mathbb{R}^{n+1}$ with $x_0 + \cdots + x_n = 0$,
modulo the lattice dual to $\mathsf{A}_n$ and
the set of complex vectors $(z_0, \ldots, z_n)$ which all lie on the unit
	circle, and whose product equals 1, modulo 
	$(\zeta, \ldots,\zeta)$ where $\zeta$ is some $(n+1)$-th root of unity
obtained by taking $(x_0, \ldots, x_n)$ to $z_r = e^{2 \pi i x_r}$
.}
or by $z_0 = \cdots = z_n = 1$.
Then, the lower bound on $\chi(\mathsf{A}_n)$ follows from
Corollary~\ref{cor:Hoffman}.
\end{proof}

\subsubsection{$\mathsf{D}_n$}

The irreducible root lattice $\mathsf{D}_n$ has $2n(n-1)$ roots
\[
R(\mathsf{D}_n) = \{\pm (e_r + e_s) : 1 \leq r,s \leq n, r \neq s\} \cup \{\pm (e_r - e_s) : 1 \leq r,s \leq n, r \neq s\}.
\]

\begin{theorem}\label{thm:lowerd}
The critical values of ${\mathcal F}_{\mathsf{D}_n}$ are:
\begin{equation*}
\left\{\begin{array}{c}
-2(n-1) \mbox{~if~}\\
n\mbox{~is~odd}
\end{array}\right\}\cup
\left\{\begin{array}{c}
2\frac{(n_1 - n_{-1})^2}{1 - n_o} - 2n_1 - 2n_{-1}\mbox{~where~}n_1, n_{-1},n_o\in \ZZ_+,\\
n_1 + n_{-1} + n_o = n \mbox{~with~} n_o = 0 \mbox{~or~}\left\vert \frac{n_1 - n_{-1}}{n_o - 1}\right\vert < 1
\end{array}\right\}.
\end{equation*}
In particular,
\[
\inf_{x\in \mathbb{R}^n} {\mathcal F}_{\mathsf{D}_n}(x) =
\begin{cases}
-2n, & \text{$n$ even},\\
-2(n-1), & \text{$n$ odd},
\end{cases}
\]
and $\chi(\mathsf{D}_n) \geq n$ for even $n$, and
$\chi(\mathsf{D}_n) \geq n+1$ for odd~$n$.
\end{theorem}

\begin{proof}
By the symmetries of the function we can restrict ourselves to $0\leq x_i\leq 1/2$.
We consider the sum
\begin{eqnarray*}
S & = & \sum_{u \in R(\mathsf{D}_n)} e^{2\pi i u \cdot x}\\
& = & 2 \sum_{1 \leq r < s \leq n} (\cos(2\pi x_r + 2\pi x_s) +
      \cos(2\pi x_r - 2\pi x_s))\\
& = & 4 \sum_{1 \leq r < s \leq n} \cos(2\pi x_r) \cos(2\pi x_s)\\
& = & 2 \left(W^2 - \sum_{r=1}^n (\cos(2\pi x_r))^2\right).
\end{eqnarray*}
where $W = \sum_{r=1}^n \cos(2\pi x_r)$. To find the critical values
we compute the gradient of~$S$ (as a function of the~$x_r$), which is
\[
\frac{\partial S}{\partial x_r} = -8 \pi \sin(2\pi x_r) \left(W - \cos(2\pi x_r) \right), \quad r = 1, \ldots, n.
\]
Let $x$ be a critical point. Define $S_{1}$, $S_{-1}$ and $S_o$ the set of $i\in \{1, \dots, n\}$ such that $x_i=0$,
$x_i=1/2$ and $0 < x_i < 1/2$ so that $\cos(2\pi x_i) = 1$ or
$\cos(2\pi x_i) = -1$ or $\cos(2\pi x_i) \in (-1,1)$.

For $i\in S_o$ we have $W = \cos(2\pi x_i)$. Thus,
$\vert W \vert < 1$ if $S_o \not= \emptyset$.

We define $n_1 = \vert S_1\vert$, $n_{-1} = \vert S_{-1}\vert$ and
$n_o = \vert S_o\vert$ and we have $n_1 + n_{-1} + n_o = n$, and
\begin{equation}
\label{eq:W}
W = \cos(0) n_1 + \cos(\pi) n_{-1} + n_o W.
\end{equation}

If $n_o=1$ we have
\[
S = 2\left(W^2 - W^2 - \sum_{r \in S_{-1} \cup S_1} (\cos(2\pi
  x_r))^2\right) = -2(n-1).
\]

If $n_o\not= 1$ then the equation for $W$ \eqref{eq:W} gives
\[
W = \frac{n_1 - n_{-1}}{1 - n_o}.
\]
The value of the function is then
\begin{equation*}
\begin{array}{rcl}
S
&=& 2 \left( W^2 - (n_1 + n_{-1} + n_o W^2)\right)\\
&=& 2 \frac{(n_1 - n_{-1})^2}{1 - n_o} - 2n_1 - 2n_{-1}.
\end{array}
\end{equation*}
The lower bound on $\chi(\mathsf{D}_n)$ then follows from
Corollary~\ref{cor:Hoffman}.
\end{proof}

\subsubsection{$\mathsf{E}_7$, $\mathsf{E}_8$ and sums of squares}
\label{ssec:e7e8sos}

We start by giving an alternative construction of the $\mathsf{E}_8$ lattice which is
based on lifting the (extended) \textit{Hamming code} $\mathcal{H}_8$, which is
the vector space over the finite field $\mathbb{F}_2$ (consisting of
the elements $0$ and $1$) spanned by the rows of the matrix
\[
G =
\left(
\begin{array}{l|l}
1000 & 0111\\
0100 & 1011\\
0010 & 1101\\
0001 & 1110\\
\end{array}
\right)
\in \mathbb{F}_2^{4 \times 8},
\]
It consists of $2^4 = 16$ code words:
\[
\begin{array}{lllll}
0000|0000 & 1000|0111 & 1100|1100 & 0111|1000 & 1111|1111\\
                  & 0100|1011 & 1010|1010 & 1011|0100 &\\
                  & 0010|1101 & 1001|1001 & 1101|0010 &\\
                  & 0001|1110 & 0110|0110 & 1110|0001 &\\
                  &                   & 0101|0101 & & \\
                  &                   & 0011|0011 & &\\
\end{array}
\]
We can define the lattice $\mathsf{E}_8$ by the following lifting
construction (which is usually called \textit{Construction A}):
\[
\mathsf{E}_8 = \left\{\frac{1}{\sqrt{2}} x : x \in \mathbb{Z}^8 ,\; x \text{ mod
    } 2 \in \mathcal{H}_8\right\}.
\]
Now it is immediate to see that $\mathsf{E}_8$ has $240$ shortest
(nonzero) vectors:
\[
\begin{array}{ll}
16 = 2^4 \text{ vectors: } & \pm \sqrt{2} e_i, \, i = 1, \ldots, 8 \\
224 = 2^4 \cdot 14 \text{ vectors: } & \frac{1}{\sqrt{2}} \sum_{j=1}^8 (\pm c_j)
                        e_j,  \; c \in \mathcal{H}_8 \text{ and }
                        \wt(c) = 4,
\end{array}
\]
where $e_1, \ldots, e_8$ are the standard basis vectors of
$\mathbb{R}^8$ and where $\wt(c) = |\{i : c_i \neq 0\}|$ denotes
the Hamming weight of a code word $c$.

Observe that the lower bound $\chi(\mathsf{E}_8) \geq 16$ is implied
trough the spectral bound by the following inequality
(Theorem~\ref{Theorem_CritValues_E8} gives a stronger result by
providing a complete list of all critical values)
\[
S(x) = \sum_{i=1}^8 2 \cos(2\pi \sqrt{2} x_i) + \sum_{c \in
  \mathcal{H}_8, \wt(c) = 4} \sum_{\pm} \cos\left(2 \pi \frac{1}{\sqrt{2}} \sum_{j=1}^8
  (\pm c_j) x_j\right) + 16 \geq 0.
\]
for all $x \in \mathbb{R}^8$. 

To simplify the formula we apply a change of variables by setting
$T(x) = S(\frac{\sqrt{2}}{2\pi} x)$. Then,
\[
T(x) = \sum_{i=1}^8 2 \cos(2 x_i) + \sum_{c \in
  \mathcal{H}_8, \wt(c) = 4} \sum_{\pm} \cos\left(\sum_{j=1}^8
  (\pm c_j) x_j\right) + 16 .
\]
Applying the cosine addition formula multiple times we get
\[
T(x) = 4 \sum_{i=1}^8 \cos(x_i)^2 + \sum_{c \in
  \mathcal{H}_8, \wt(c) = 4} 16 \prod_{j=1}^8 \cos(c_j x_j).
\]
Function $T$ is globally nonnegative if and only if the polynomial
\[
p(t) = \sum_{i=1}^8 t_i^2 + 4\sum_{c \in \mathcal{H}_8, \wt(c) = 4, \supp c =
\{i,j,k,l\}} t_i t_j t_k t_l
\]
is nonnegative on the cube $t \in [-1,+1]^8$. This nonnegativity can
be verified algorithmically. By
\[
\begin{split}
\Sigma_{n,d} = \{p \in \mathbb{R}[x_1, \ldots, x_n] : & \deg p \leq
d, \; \text{there are } p_1, \ldots, p_m \in \mathbb{R}[x_1, \ldots, x_n] :\\
& p = p_1^2+ \cdots + p_m^2\}
\end{split}
\]
denote the cone of real polynomials in $n$ indeterminates of degree at
most $d$. One can verify membership in this cone by using semidefinite
optimization, see for example \cite{Lasserre2015a}. We checked
numerically (up to machine precision)
that there are polynomials
\[
q \in \Sigma_{8,8}, q_1, \ldots, q_8 \in \Sigma_{8,6}
\]
so that the following identity holds true:
\[
p(t) = q(t) + \sum_{i=1}^8 (1-t_i^2) q_i(t).
\]
It is interesting to observe that using smaller degree did not work.

One can easily modify this proof technique for the case $\mathsf{E}_7$
to show $\chi(\mathsf{E}_7) \geq 10$. To define $\mathsf{E}_7$ we
apply Construction A on the $[7,3,4]$ code $\mathcal{H}_7^*$, which
one obtains from $\mathcal{H}_8$ by deleting its first
coordinate. Here one shows that the polynomial
\[
\sum_{i=1}^7 t_i^2 + 4\sum_{c \in \mathcal{H}_7^*, \wt(c) = 4, \supp c =
\{i,j,k,l\}} t_i t_j t_k t_l
\]
is nonnegative on the cube $t \in [-1,+1]^7$ again using sum of squares.

\subsubsection{$\mathsf{E}_6$}
\label{ssec:e6Schlafi}

The $\mathsf{E}_6$ lattice we can handle without computer as follows:
In the proof of Theorem \ref{thm:e} we shall color $\mathsf{E}_6$ with
nine colors. On the other hand, the Schl\"afli polytope is a Delaunay
polytope of $\mathsf{E}_6$ whose vertex-edge graph---the Schl\"afli
graph on 27 vertices having 216 edges---is a finite subgraph of
$\Cayley(\mathsf{E}_6, \Vor(\mathsf{E}_6))$. It is known, see for
example \cite[Chapter 8, page 55]{Biggs1974a} that the Hoffman bound
of the Schl\"afli graph equals nine. It is also known, see
\cite[Section 10.1]{Bachoc2009a}, that the Hoffman bound of an
infinite edge transitive graph is at least the Hoffman bound of any of
its finite subgraphs. Hence, the spectral bound of $\mathsf{E}_6$
equals nine.

\section{The chromatic number of irreducible root lattices and their duals}
\label{sec:coloring-root}

In this section, we complete our study of the chromatic number of
irreducible root lattices and their duals.  The knowledge that we use
about Delaunay polytopes of root lattices can be found in
\cite[Section 14.3]{Deza1997a}.  Our claims regarding the sublattices
that we consider and the colorings of certain small graphs can be
conveniently checked with computer assistance, for example by using
{\tt Magma} \cite{Magma} or {\tt Polyhedral} \cite{Polyhedral}.

When we cannot directly compute the chromatic number of a graph, we
apply other methods, computationally easier, in order to get lower and
upper bounds.  A lower bound for the chromatic number of a graph
$G = (V,E)$ is given by its \textit{fractional chromatic number}: Denote
by $\mathcal{I}_G$ the set of all independent sets of $G$. The
fractional chromatic number of $G$ is the solution of the following
linear program:

\[
\min\left\{\sum_{I \in \mathcal{I}_G} \lambda_I : \lambda_I \in \mathbb{R}_{\geq 0} \text{ for } I \in \mathcal{I}_G, \sum_{I \in \mathcal{I}_G \text{ with } v\in I } \lambda_I \geq 1 \text{ for } v \in V \right\}.
\]
If $G$ affords symmetries, one can use them to reduce the number of
variables of this linear program.  Regarding upper bounds, given a
number $k$ of colors and a graph $G$, proving the existence of a
$k$-coloring of $G$ can be turned into a satisfiability problem, that
can be solved for instance by using {\tt Minisat} \cite{Minisat}.

\subsection{$\mathsf{D}_n$ and its dual}

The half cube $\frac{1}{2} H_n$, sometimes also called the parity polytope, is defined as
\begin{equation*}
\frac{1}{2} H_n = \conv\left\{ x\in \{0,1\}^n : \sum_i x_i = 0 \bmod 2 \right\} .
\end{equation*}

It is one of the two Delaunay polytopes of the root lattice
$\mathsf{D}_n$.

\begin{theorem}\label{thm:d}
  For all $n \geq 4$ we have
  $\chi(\mathsf{D}_n) = \chi(\frac{1}{2} H_n)$, where we consider the
  vertex-edge graph of the half cube.
\end{theorem}

\begin{proof}
  The inequality $\chi(\mathsf{D}_n) \geq \chi(\frac{1}{2} H_n)$ comes
  from the fact that $\frac{1}{2} H_n$ is a Delaunay polytope of
  $\chi(\mathsf{D}_n)$.

  Let $c$ be a proper coloring of $\frac{1}{2} H_n$. We extend it to
  $\mathsf{D}_n$ by giving to any $x \in \mathsf{D}_n$ the color
  $c(x \bmod 2)$.  Assume that two vectors $x_1$ and $x_2$ are
  adjacent in $\Cayley(\mathsf{D}_n, \Vor(\mathsf{D}_n))$.  Since the
  relevant vectors of $\mathsf{D}_n$ are of the form
  $\pm e_i \pm e_j$, the difference $x_1 - x_2 \bmod 2$ is also such a
  vector, so that $x_1 \bmod 2$ and $x_2 \bmod 2$ are adjacent in
  $\frac{1}{2} H_n$, and $c(x_1 \bmod 2) \neq c(x_2 \bmod 2) $. Hence,
  $\chi(\mathsf{D}_n) \leq \chi(\frac{1}{2} H_n)$.
\end{proof}

\begin{theorem}
  For every $n \geq 4$, the chromatic number of $\mathsf{D}^*_n$ is
  $4$.
\end{theorem}

\begin{proof}
  The relevant vectors of the lattice
  $\mathsf{D}_n^* = \mathbb{Z}^n \cup \left((1/2, \ldots, 1/2) +
    \mathbb{Z}^n \right)$
  are the $2n$ vectors $\pm e_i$ and the $2^n$ vectors of the form
  $(\pm 1/2, \dots, \pm 1/2)$. The four vectors $0$,
  $(1,0, \ldots, 0)$, $(1/2, \ldots ,1/2)$ and $(1/2,\ldots,1/2,-1/2)$
  define a clique in $\Cayley(\mathsf{D}_n^*, \Vor(\mathsf{D}_n^*))$;
  and the unique way to color $\mathsf{D}^*_n$ with four colors is by
  coloring each copy of $\mathbb{Z}^n$ with two different colors.
\end{proof}

\subsection{$\mathsf{E}_6, \mathsf{E}_7, \mathsf{E}_8$ and their duals}

\begin{theorem}\label{thm:e}
  We have $\chi(\mathsf{E}_6) = 9$, $\chi(\mathsf{E}_7) = 14$ and
  $\chi(\mathsf{E}_8) = 16$.
\end{theorem}

\begin{proof}
By Theorem~\ref{thm:d} we know that $\chi(\mathsf{D}_8) = 8$.
So one can color the root lattice $\mathsf{E}_8 = \mathsf{D}_8 \cup ((1/2, \ldots, 1/2)) +  \mathsf{D}_8)$
with $16$ colors. The lower bound from Theorem~\ref{Theorem_CritValues_E8}
concludes the case of $\mathsf{E}_8$.

For $\mathsf{E}_7$ there are two orbits of Delaunay polytopes. One of them is the Gosset
polytope with $56$ vertices, whose vertex-edge graph has chromatic number $14$,
which shows that $\chi(\mathsf{E}_7) \geq 14$. Moreover, we
have a lamination of $\mathsf{E}_7$ over the lattice $\mathsf{A}_6$:
\begin{equation*}
\mathsf{E}_7 = \bigcup_{n\in \mathbb{Z}} (n u + \mathsf{A}_6) \mbox{~for~some~} u\in \mathsf{E}_7.
\end{equation*}
If $v$ and $w$ belong to two layers which differ by an even index, then $v - w$
is not a relevant vector. Following Section \ref{sec:firstkind},
we know that $\chi(\mathsf{A}_6) = 7$. Thus we can color
the even layers by colors in $\{1, \dots 7\}$ and the odd layers by
colors in $\{8,\dots, 14\}$. This proves that
$\chi(\mathsf{E}_7) = 14$.

The unique Delaunay polytope of $\mathsf{E}_6$ is the Schl\"afli
polytope whose vertex-edge graph is the
Schl\"afli graph with $27$ vertices. It is well known that its
chromatic number is $9$, so that $\chi(\mathsf{E}_6) \geq 9$.
There is just one orbit of independent triples of vertices.
Moreover, there are just two orbits of colorings
with $9$ colors of the Sch\"afli graph: One orbit of size $160$ and
another of size $40$.  Let us take a coloring from the orbit of size
$40$. It is composed of $9$ different triples of elements.  For each
such triple $\{v_1, v_2, v_3\}$ we consider the set of vectors
$\{v_1-v_2, v_2 - v_3, v_3 - v_1\}$.  Since we have $9$ triples, this
gives in total $27$ vectors. Those vectors span a sublattice $L$ of
$\mathsf{E}_6$ of index $9$. It turns out that none of the relevant
vectors of $\mathsf{E}_6$ belongs to $L$. Thus following Lemma~\ref{lem:upperboundquotient}, $\chi(\mathsf{E}_6) \leq 9$.
\end{proof}

\begin{theorem}
\label{thm:chromaticdual}
The chromatic number of $\mathsf{E}_n^*$ is $16$ for $n=6,7$.
\end{theorem}

\begin{proof}
  Upon rescaling to an integral lattice the norms of the vectors of
  $\mathsf{E}_6^*$ are $4$, $6$, $10$ and so on.  The norms of the
  relevant vectors are $4$ and $6$. We consider the $432$ vectors of
  norm $10$ and enumerate the sublattices of $\mathsf{E}_6^*$ of
  dimension $6$ spanned by those vectors that do not contain any
  relevant vector.  We found $1393$ orbits of such lattices. Exactly
  one of them is of index $16$ which proves that
  $\chi(\mathsf{E}_6^*) \leq 16$.

  The lower bound is obtained in the following way. We consider the
  graph formed by the origin $0$ and the $126$ relevant vectors with
  two vectors adjacent if their difference is a relevant vector. The
  fractional chromatic number of this graph is $77/5$. Thus
  $\chi(\mathsf{E}_6^*) \geq \lceil 77 / 5\rceil = 16$.

  The lower bound on the chromatic number of $\mathsf{E}_7^*$ is
  obtained by the same technique as for $\mathsf{E}_6^*$.  The upper
  bound is obtained in the following way: Consider the quotient
  $\mathsf{E}_7^* / 4\mathsf{E}_7^*$ with $16384$ elements. One
  coloring with $16$ colors can be obtained by solving the
  satisfiability problem using {\tt Minisat}.
\end{proof}

\section*{Acknowledgements}

We thank Jean-Pierre Serre for his patient explanation of how to
compute the critical points of the function
$\ch_{\ad}^{\mathsf{E}_n}$.  We are indebted to Christine Bachoc and
Gabriele Nebe for their important contributions to
Section~\ref{sec:spherepackingbound}. We thank the anonymous
referee for helpful comments and suggestions.

This material is based upon work supported by the National Science
Foundation under Grant No.~DMS-1439786 while M.D.S.\ and P.M.\ were in
residence at the Institute for Computational and Experimental Research
in Mathematics in Providence, RI, during the ``Point Configurations in
Geometry, Physics and Computer Science'' semester program. P.M.\
received support from the Troms{\o} Research Foundation grant
17\_matte\_CR. F.V.\ was partially supported by the SFB/TRR 191
``Symplectic Structures in Geometry, Algebra and Dynamics''. F.V.\
also gratefully acknowledges support by DFG grant VA 359/1-1. This
project has received funding from the European Union's Horizon 2020
research and innovation programme under the Marie Sk\l{}odowska-Curie
agreement number~764759.

\appendix

\section{Recollections on compact Lie groups}
\label{sec:liegroups}

In this section, we collect, for the benefit of the unfamiliar reader,
without proof but with references, a few facts about semisimple
compact Lie groups and representation theory.  All of the following
results are standard, although it is not easy to find them
conveniently stated in a single place, so we hope this compendium will
be helpful.

The main reason for this appendix insofar as the present paper is
concerned is to explain the reason behind the reformulation which we
give in theorem~\ref{thm:serrestheorem} of Serre's \cite[Theorem
  3']{Serre2004a}, namely the connection between the Fourier transform
of a root system $\Phi$ and the character of the adjoint
representation of the Lie groups associated with~$\Phi$: this is
provided by~\ref{prop:adjcharvals}.  We have, however, stated a few
additional results which are not strictly necessary towards that goal
but which, we hope, help give a clearer overall picture.  The
secondary reason for this appendix is to provide the necessary
framework for appendix~\ref{sec:serresproof} (although the latter
could, in principle, be reworded so as to eliminate all mentions of
Lie groups just like we did for theorem~\ref{thm:serrestheorem}, we
believe this would be unnecessarily contrived).

\begin{remark}
We have chosen to focus these recollections on semisimple compact real
Lie groups, but the classification and representation theory of
semisimple \emph{complex} Lie groups is identical (Weyl's ``unitarian
trick'', cf.~\cite[\S6.2]{Kirillov2008} and
\cite[\S26]{FultonHarris2004}): we simply mention that the role of the
tangent Lie algebra $\mathfrak{t}$ to a maximal torus in what follows
is played, in the complex setting, by Cartan subalgebras
$\mathfrak{h}$ of~$\mathfrak{g}$ (\cite[definition~6.32]{Kirillov2008}
and \cite[\S14.1 and appendix~D]{FultonHarris2004}).
\end{remark}

%% Among the many standard references where these results can be found
%% are \cite[esp.~\S19 and~\S23--24]{Bump2004}, \cite[esp.~\S14, \S21
%%   and~\S23]{FultonHarris2004}, \cite{Kirillov2008} or \cite[esp.~\S13
%%   and~\S20]{Humphreys1972}.

\parapoint\label{par:cptliegroups}
A \emph{compact (real) Lie group} is a compact connected real
smooth manifold~$G$, together with a group structure on~$G$
such that the multiplication and inverse maps are smooth ($C^\infty$).
The tangent space $\mathfrak{g}$ at the identity~$1$ of~$G$ is then
endowed with a linear action $\Adjt$ of~$G$, called the \emph{adjoint
  representation} of~$G$, defined by letting $\Adjt(g) \colon
\mathfrak{g} \to \mathfrak{g}$ (for~$g\in G$) be the differential
at~$1$ of $u\mapsto g u g^{-1}$; this in turns defines a linear map
$\adjt(x) \colon \mathfrak{g} \to \mathfrak{g}$ for $x\in\mathfrak{g}$
by letting $\adjt \colon \mathfrak{g} \to L(\mathfrak{g},
\mathfrak{g})$ (where $L(U,V)$ stands for the vector space of linear
maps between two vector spaces $U$~and~$V$) be the differential at~$1$
of~$\Adjt \colon G \to L(\mathfrak{g}, \mathfrak{g})$ itself
(see~\cite[\S8.1]{FultonHarris2004}): writing $[x,y]$ for
$\adjt(x)(y)$, this gives $\mathfrak{g}$ the structure of a (real)
\emph{Lie algebra} (simply known as the Lie algebra of~$G$).

We also recall the exponential map $\exp \colon \mathfrak{g} \to G$,
which takes $x\in\mathfrak{g}$ to the value at~$1$ of the unique
smooth group homomorphism $\mathbb{R} \to G$ (a.k.a. ``$1$-parameter
subgroup'') whose differential at the origin is~$x$.  While $\exp$ is
not a group homomorphism in general, it is one whenever $G$ is abelian
(whenever the Lie bracket of $\mathfrak{g}$ vanishes; otherwise, the
so-called Baker-Campbell-Hausdorff formula expresses the relation of
$\exp(x)\,\exp(y)$ to an exponential).  Rather than using the
exponential, we will find it more convenient to use the function
$\mathbf{e} \colon x \mapsto \exp(2\pi x)$; just as $\exp$ itself, the
function $\mathbf{e}$ in question is surjective
(cf.~\ref{par:toruslie} below).

The \emph{Killing form} $B \colon \mathfrak{g}\times\mathfrak{g} \to
\mathfrak{g}$ is the bilinear form $B(x,y) =
\tr(\adjt(x)\circ\adjt(y))$, which is $G$-invariant; in the real
compact case in which we placed ourselves, this form is negative
semidefinite (\cite[theorem~6.10]{Kirillov2008} or
\cite[theorem~2.13]{Arvanitoyeorgos2003}), and we say that $G$ or
$\mathfrak{g}$ is \emph{semisimple} when $B$ is nondegenerate
(\cite[prop.~C10]{FultonHarris2004},
\cite[theorem~5.53]{Kirillov2008}), i.e., negative definite in the
real compact case.

As an example, the group $\mathit{SO}_{2n}$ of $(2n)\times(2n)$ real
orthogonal matrices with determinant~$+1$ is a compact real Lie group,
whose Lie algebra $\mathfrak{so}_{2n}$ consists of antisymmetric
$(2n)\times(2n)$ matrices, the Lie bracket $[x,y]$ being the usual
$x y-y x$, and the Killing form on $\mathfrak{so}_{2n}$ is given by
$B(x,y) = 2(n-1)\tr(x y)$, so that $\mathfrak{so}_{2n}$ is semisimple
if and only if~$n\geq 2$.

\parapoint\label{par:liegroupsvsalgs}
If $G$ is a compact Lie group with Lie algebra~$\mathfrak{g}$, the map
taking a connected closed subgroup $H$ of $G$ to its Lie algebra seen
as a subalgebra $\mathfrak{h}$ of~$\mathfrak{g}$ (i.e., a vector
subspace closed under the Lie bracket) is a bijection
(\cite[theorem~3.40]{Kirillov2008}).

Two \emph{simply connected} compact Lie groups are isomorphic if and
only if their Lie algebras are isomorphic
(\cite[theorem~3.43]{Kirillov2008}); thus, two compact Lie groups with
isomorphic Lie algebras have isomorphic universal coverings: they are
then said to be \emph{isogenous}.  (We note that, for the purposes of
this paper, isogenous Lie groups are an irrelevant complication.)
Beware, however, that the universal covering of a compact Lie group
need not be compact as the case of tori shows.

Conversely, any real Lie algebra with a negative semidefinite Killing
form is the Lie algebra of some compact Lie group (unique up to
isogeny, by the previous paragraph).  We return
in~\ref{par:liefundgrp} to the question of which Lie groups are
possible in the semisimple case.

\parapoint\label{par:maxtori}
If $G$ is a compact Lie group, a \emph{torus} in~$G$ is an abelian
connected closed subgroup of~$G$, or equivalently, one whose Lie
algebra is abelian (meaning that its Lie bracket is trivial).  A
\emph{maximal torus}, of course, is a torus that is maximal for
inclusion; by \ref{par:liegroupsvsalgs}, maximal tori of $G$ are in
bijection with maximal abelian Lie subalgebras of the Lie
algebra~$\mathfrak{g}$ of~$G$.

As an example, a maximal torus in $\mathit{SO}_{2n}$ is given by the
block diagonal matrices whose diagonal blocks are $2\times 2$ rotation
matrices.

Crucial results by Cartan concerning maximal tori of compact Lie
groups are (\cite[theorems 16.4~and~16.5]{Bump2004} or
\cite[theorem~2.15]{Arvanitoyeorgos2003} or \cite[corollaries
  4.35~and~4.46]{Knapp2002}): \textbf{(a)}~every element of $G$
belongs to some maximal torus, and \textbf{(b)}~all maximal tori of
$G$ are conjugate; in particular, each element of $G$ is conjugate to
some element of any fixed maximal torus of~$G$.

The dimension of some (any) maximal torus~$T$ in $G$ (or equivalently,
of its Lie algebra) is known as the \emph{rank} of~$G$.  The quotient
$N_G(T)/T$ of the normalizer of $T$ (in~$G$) by~$T$ itself is known as
the \emph{Weyl group} $W$ of~$G$: so a $W$-orbit in $T$ is precisely a
full set of $G$-conjugate elements of~$T$, and the set of conjugacy
classes in $G$ can be identified (as a set) with~$T/W$.  We note that
$W$ acts as a group of automorphisms of~$T$, so it also acts
(linearly) on the Lie algebra $\mathfrak{t}$ of~$T$ and (by inverse
transpose) on the dual $\mathfrak{t}^*$ of~$\mathfrak{t}$.

\parapoint\label{par:toruslie}
If $T$ is an abstract torus, i.e., an abelian compact Lie group, and
$\mathfrak{t}$ is its Lie algebra, the map $x \mapsto \mathbf{e}(x) :=
\exp(2\pi x)$ (in other words the differentiable group homomorphism
$\mathfrak{t} \to T$ whose differential at~$0$ is $2\pi$ times the
identity) defines a surjective homomorphism $\mathfrak{t} \to T$,
whose kernel is a discrete subgroup $\Gamma$ of~$\mathfrak{t}$.  Thus,
we can identify $T$ with $\mathfrak{t}/\Gamma$ (as a differentiable
group), and functions on $T$ with $\Gamma$-periodic functions
on~$\mathfrak{t}$.  We will call $\Gamma$ the \emph{period lattice} of
the torus~$T$.

\parapoint\label{par:lierepns}
If $G$ is a compact Lie group, a (finite-dimensional)
\emph{representation} of~$G$ on a finite-dimensional complex vector
space~$V$ is a differentiable linear action of $G$ on~$V$, i.e., a
differentiable group homomorphism $\rho\colon G \to \mathit{GL}(V)$.
The \emph{character} of said representation is the map $g \mapsto
\tr\rho(g)$ (a differentiable function on~$G$, invariant under
conjugation).  The representation is said to be \emph{irreducible}
when the only $G$-invariant subspaces of~$V$ are $0$~and~$V$.  It
turns out that every representation of~$G$ is a direct sum of
irreducible representations (\cite[theorem~4.40]{Kirillov2008}); and a
representation is characterized (up to isomorphism) by its character
(\cite[theorem~2.5]{Bump2004} or \cite[theorem~4.46]{Kirillov2008}).

Furthermore, although we will not use this, it might be worth pointing
out the Peter-Weyl theorem: the characters of the irreducible
representations of~$G$ form a Hilbert orthonormal basis for the closed
subspace consisting of conjugation-invariant functions in the Hilbert
space $L^2(G)$ of square-integrable functions on~$G$
(\cite[theorem~4.50]{Kirillov2008}; incidentally, these functions are
also eigenvalues of the Laplace-Beltrami operator on $G$ seen as a
Riemannian manifold).

Among the representations of~$G$, the adjoint representation (defined
in~\ref{par:cptliegroups} above as a map $\Adjt\colon
G\to\mathit{GL}(\mathfrak{g})$, which we see as an action on the
complexified vector space $\mathfrak{g}_{\mathbb{C}} := \mathfrak{g}
\otimes_{\mathbb{R}} \mathbb{C}$) is of particular importance; its
character $g \mapsto \tr\Adjt(g)$ is called the character of the
adjoint representation, or simply the adjoint character, of~$G$; the
adjoint representation is irreducible if and only if $G$ is simple
(this can be taken as a definition\footnote{A Lie group $G$ is called
  ``simple'' when it does not have nontrivial \emph{connected} normal
  subgroups: this allows for a finite center (the term ``quasisimple''
  might be more appropriate).}).

\parapoint\label{par:torusrepns}
Representation theory on a torus $T$ is well known: writing $T =
\mathfrak{t}/\Gamma$ through $x \mapsto \mathbf{e}(x) := \exp(2\pi x)$
as in~\ref{par:toruslie}, the irreducible characters of~$T$ are of the
form $\mathbf{e}(\lambda) : \mathbf{e}(x) \mapsto \exp(2\pi i
\lambda(x))$ with $\lambda$ ranging over the lattice $\Gamma^*$ dual
to $\Gamma$ (in the vector space $\mathfrak{t}^*$ dual to
$\mathfrak{t}$), which we call the \emph{character lattice} of~$T$.
The corresponding representations are all one-dimensional (acting by
multiplication by the character just defined).  In other words, the
irreducible representations of $T$ are indexed by~$\Gamma^*$, and the
orthonormal basis of characters of $T$ predicted by the Peter-Weyl
theorem is the usual Fourier basis on~$T = \mathfrak{t}/\Gamma$.  (As
mentioned in~\ref{par:toruslie}, we write $\mathbf{e}(\lambda)$ both
for the function $T \to \mathbb{C}$ defined earlier, and for the
$\Gamma$-periodic function $\mathfrak{t} \to \mathbb{C}$ given by $x
\mapsto \exp(2\pi i \lambda(x))$.)

We note that the ring of linear combinations (over $\mathbb{Z}$,
resp.~$\mathbb{C}$) of the $\mathbf{e}(\lambda)$, or \emph{character
  ring} (resp. character $\mathbb{C}$-algebra) of~$T$
(cf.~\ref{par:charring} below), is the group ring (resp.~group
$\mathbb{C}$-algebra) of the lattice~$\Gamma^*$.  (Choosing a basis
for $\Gamma^*$ shows that this is a ring of Laurent polynomials.)

\parapoint\label{par:charweights}
If now $G$ is a compact Lie group and $T$ is a maximal
torus in~$G$ with corresponding Lie algebras $\mathfrak{t} \subseteq
\mathfrak{g}$, given a representation $V$ of~$G$ having character
$\chi$, we can restrict them to~$T$ and consider the Fourier
decomposition of~$\chi|_T$, i.e., its decomposition $\chi|_T =
\sum_{\lambda\in\Gamma^*} m_\lambda \mathbf{e}(\lambda)$ in terms of
the characters $\mathbf{e}(\lambda)$ (for $\lambda\in\Gamma^*$)
defined in the previous paragraph: clearly $m_\lambda$ is the
dimension of the subspace $V^{\lambda}$ of $V$ consisting of those
$z\in V$ such that $\rho(u)(z) = \mathbf{e}(\lambda)(u)$ for each
$u\in T$.  In particular, $m_\lambda\in\mathbb{N}$.  The $\lambda$
such that $m_\lambda>0$ are known as the \emph{weights} of the
representation~$V$ (or of the character~$\chi$), and we emphasize that
they belong to~$\Gamma^*$; the value $m_\lambda$ is know as the
\emph{multiplicity} of the weight~$\lambda$ (in $V$ or in~$\chi$), and
the subspace $V^{\lambda}$ on which $T$ acts through
$\mathbf{e}(\lambda)$ is known as the \emph{weight (eigen)space}; we
note that the weight space can be defined at the Lie algebra level as
the set of $z$ such that $d\rho_1(x)(z) = i\lambda(x)\,z$ for all
$x\in \mathfrak{t}$ (where $d\rho_1$ is the differential of~$\rho$
at~$1$; compare \cite[definition~8.1]{Kirillov2008}: we have added a
factor $i$ here for convenience in the compact case, but it is a
matter of convention).

Since, as explained in~\ref{par:maxtori} above, all maximal tori
of~$G$ are conjugate, the weights and multiplicities do not depend on
the choice of~$T$; furthermore, they are invariant under the action of
the Weyl group~$W$.  Seen as a function on $\mathfrak{t}$, the
character values are both $\Gamma$-periodic and $W$-invariant, so they
are $\Gamma \rtimes W$-invariant.  (Let us mention here the paper
\cite{Patera2004}, which can serve as link between the ``Fourier'' and
``Lie group characters'' points of view.)

\parapoint\label{par:rootsystems}
We temporarily leave aside Lie groups to recall the following
definitions and facts in relation with abstract root systems (see
\cite[\S9.2]{Humphreys1972}, \cite[\S7.1]{Kirillov2008} and
\cite[chap.~VI]{BourbakiLIE}).  A (reduced, crystallographic)
\emph{root system} is a set $\Phi$ of vectors in a finite dimensional
real vector space~$E$ such that (1)~$\Phi$ is finite, does not
contain~$0$, and spans~$E$, (2)~for every $\alpha\in\Phi$, there
exists $\alpha^\vee$ in the dual space $E^*$ of~$E$ such that
$\alpha^\vee(\alpha) = 2$ and such that the (symmetry) map
$s_\alpha\colon x \mapsto x - \alpha^\vee(x)\,\alpha$ leaves~$\Phi$
stable (it is easy to see that $\alpha^\vee$ is uniquely defined,
cf.~\cite[chap.~VI, \S1, n\textsuperscript{o}1, lemme~1]{BourbakiLIE},
so that the notation is legitimate), (3)~for every
$\alpha,\beta\in\Phi$ we have $\alpha^\vee(\beta) \in\mathbb{Z}$, and
(4)~if $\alpha\in\Phi$ and $c\alpha\in\Phi$ then $c \in \{\pm 1\}$.
The elements of $\Phi$ are called \emph{roots}, and the $\alpha^\vee$
are the \emph{coroots}.  The set $\Phi^\vee := \{\alpha^\vee :
\alpha\in\Phi\}$ of coroots is itself a root system (with
$(\alpha^\vee)^\vee = \alpha$), known as the \emph{dual} root system
to~$\Phi$.  The group generated by the $s_\alpha$ is known as the
\emph{Weyl group} of~$\Phi$, and it is finite.

Two root systems $\Phi\subseteq E$ and $\Phi'\subseteq E'$ are said to
be \emph{isomorphic} when there is a linear isomorphism between $E$
and~$E'$ taking $\Phi$ to~$\Phi'$.  In this case, they have isomorphic
dual systems and isomorphic Weyl groups.

The root system $\Phi\subseteq E$ is said to be \emph{reducible} when
it is the union (``sum'') of root systems $\Phi_1,\Phi_2$ in $E_1,E_2$
with $E = E_1 \oplus E_2$ respectively, \emph{irreducible} otherwise.
Every root system can be written in a unique way as the sum of
irreducible root systems, and any sum of root systems is a root
systems.

Given a root system $\Phi$ in~$E$, there exists a Euclidean structure
on~$E$ such that every~$s_\alpha$ (and consequently, every element of
the Weyl group) is orthogonal; equivalently, a Euclidean structure
which identifies $E$ with its dual~$E^*$ so that each coroot
$\alpha^\vee$ is proportional to the corresponding root~$\alpha$.
Such a Euclidean structure is said to be \emph{compatible}
with~$\Phi$.  (The definition of root systems is often written in a
manner that preassumes the Euclidean structure: in this case, the
coroot $\alpha^\vee$ associated to~$\alpha$ is defined as
$2\alpha/\|\alpha\|^2$.)  For $\Phi$ irreducible, this Euclidean
structure is unique up to a multiplicative constant, i.e., up to the
definition of the lengths of the roots; in the case of the simply
laced root system (those in which every root has the same length),
which concerns us in the present paper, the constant is generally
chosen such that the squared root length is~$2$, so that $\Phi$
and~$\Phi^\vee$ can be identified.  Nevertheless, it might be useful
for expositional clarity to keep the distinction between $\Phi$
and~$\Phi^\vee$ and the greater generality afforded by the not
necessarily simply laced root system, so we do not perform this
identification (but the reader may choose to do so).

Associated with a root system $\Phi$ as above are four lattices: the
lattice $Q := \mathbb{Z}\Phi \subseteq E$ generated by~$\Phi$ is known
as the \emph{root lattice}, the lattice $Q^\vee := \mathbb{Z}\Phi^\vee
\subseteq E^*$ generated by~$\Phi^\vee$ is known as the \emph{coroot
  lattice}; the lattice $P := (Q^\vee)^*$ (in~$E$) dual to the coroot
lattice is known as the \emph{weight lattice} and contains the root
lattice; and the lattice $P^\vee := Q^*$ (in~$E^*$) dual to the root
lattice is known as the \emph{coweight lattice} and contains the
coroot lattice.  The quotient of the weight lattice by the root
lattice, or equivalently of the coweight lattice by the coroot
lattice, is sometimes known as the \emph{fundamental group} of~$\Phi$
for reasons that will be clarified in~\ref{par:liefundgrp}.

\parapoint\label{par:posroots}
Continuing the exposition of root systems started
in~\ref{par:rootsystems}, if $h$ is a linear form on~$E$ such that
$h(\alpha)\neq 0$ for each~$\alpha\in\Phi$, the roots such that
$h(\alpha)>0$ are then known as the \emph{positive roots}, and those
such that $h(\alpha)<0$ as the \emph{negative roots} relative to~$h$:
a subset $\Phi_+ := \{\alpha\in\Phi : h(\alpha)>0\}$ which can be
obtained in this manner is known as a \emph{choice of positive roots}
for~$\Phi$.  The positive roots which cannot be written as sums of
other positive roots are known as \emph{simple roots} (for this choice
of positive roots): it is then a fact that the simple roots form a
basis of~$E$, and that every positive root is a linear combination of
the simple roots with nonnegative integer coefficients (not all zero);
so the choice of positive roots can be defined equivalently by the set
of simple roots.  The Weyl group acts simply transitively on the set
of all choices of positive roots (or the set of all choices of simple
roots).

The choice of positive roots also gives a choice of positive coroots,
(defined from $h$ as above by identifying $E$~with~$E^*$, or by saying
that the positive coroots are the coroots associated with positive
roots).  The dual basis $\varpi_1,\ldots,\varpi_n$ to the set
$\alpha^\vee_1,\ldots,\alpha^\vee_n$ of simple coroots is known as the
set of \emph{fundamental weights} (for the choice of positive roots);
symmetrically, the dual basis to the set of simple roots is known as
the set of \emph{fundamental coweights}.

The convex cone in $E$ generated by the fundamental weights, is known
as the \emph{closed Weyl chamber} in~$E$ corresponding to the choice
of positive roots: it is the dual cone to the positive coroot cone, in
other words, it is defined by the inequalities $\alpha^\vee(x) \geq 0$
for all positive coroots (or equivalently, for all simple coroots)
$\alpha^\vee$; the open Weyl chamber, defined by the inequalities
$\alpha^\vee(x) > 0$, is the interior of the closed Weyl chamber (and
the closed Weyl chamber is its closure).  Dually, the cone in~$E^*$
generated by the fundamental coweights, which is the dual cone to the
positive root cone, is also known as the closed Weyl chamber
(in~$E^*$).  The choice of a Weyl chamber is equivalent to a choice of
positive roots: the Weyl group acts simply transitively on the set of
Weyl chambers.

Given a choice of positive roots $\Phi_+ \subseteq \Phi$, there exists
a unique $\beta \in \Phi_+$ such that $\beta+\alpha \not\in \Phi_+$
for all~$\alpha \in \Phi_+$: this is known as the \emph{highest root}
of~$\Phi$ (relative to this choice of positive roots), and it belongs
to the open Weyl chamber.  The (integer) coefficients $m_i$ of $\beta$
on the basis of simple roots, i.e., the $m_i$ such that $\beta =
\sum_{i=1}^n m_i \alpha_i$ where $\alpha_1,\ldots,\alpha_n$ are
the simple roots, often come up in formulae involving $G$ or its root
system.  It is often more convenient to define $\alpha_0 = -\beta$
(the \emph{lowest root}) and $m_0 = 1$ so that $\sum_{i=0}^n m_i
\alpha_i = 0$.

\parapoint\label{par:dynkindiags}
Given a root system $\Phi$ and a choice of positive roots, we define
the \emph{Dynkin diagram} of~$\Phi$ as the graph whose vertices
(``nodes'') are the simple roots, two nodes $\alpha,\beta$ being
connected by a single, double or triple edge, or by no edge at all,
according as the angle between them is $2\pi/3$, $3\pi/4$ or $5\pi/6$,
or $\pi/2$ for no edge at all, these being the only possible values;
in the case of a double or triple edge, it is oriented by pointing
from the simple root with the larger norm to that with the smaller
norm.  (These constructions rely on a Euclidean structure compatible
with~$\Phi$, but are independent of the choice of such a structure.)

The Dynkin diagram of a root system $\Phi$ determines the latter up to
isomorphism.  Furthermore, all possible irreducible root systems can
be classified, with the help of Dynkin diagrams. The simply laced
Dynkin diagrams are shown on
Figure~\ref{FigureDynkinDiagramNumbering}.

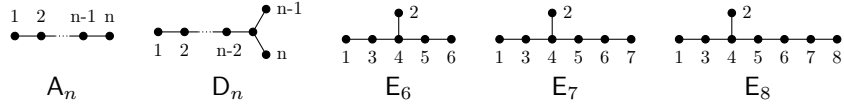
\begin{figure}
\begin{tabular}{ccccc}
\begin{tikzpicture}
 \dynkin[]{A}{};
\dynkinLabelRoot*{1}{1}
\dynkinLabelRoot*{2}{2}
\dynkinLabelRoot*{3}{$n-1$}
\dynkinLabelRoot*{4}{$n$}
\end{tikzpicture}
 &
\raisebox{-0.32cm}{\begin{tikzpicture}

 \dynkin[]{D}{};
\dynkinLabelRoot{1}{1}
\dynkinLabelRoot{2}{2}
\dynkinLabelRoot{4}{$n-2$}
\dynkinLabelRoot{5}{$n-1$}
\dynkinLabelRoot{6}{$n$}
\end{tikzpicture}}
&
\dynkin[label]{E}{6}
&
\dynkin[label]{E}{7}
&
\dynkin[label]{E}{8}
\\
$\mathsf{A}_n$ & $\mathsf{D}_n$ & $\mathsf{E}_6$ & $\mathsf{E}_7$ & $\mathsf{E}_8$
\end{tabular}
\caption{Simply laced Dynkin diagrams with the Bourbaki numbering of their nodes}
\label{FigureDynkinDiagramNumbering}
\end{figure}

\parapoint\label{par:lieroots}
We now return to the setup of a compact Lie group $G$ as
in~\ref{par:charweights}, and we furthermore assume $G$ to be
semisimple.

The nonzero weights of the adjoint representation of~$G$ are known as
the \emph{roots} of~$G$ (or of~$\mathfrak{g}$), and each one occurs
with multiplicity~$1$; as for the zero weight space, it is the
complexification of $\mathfrak{t}$ itself (in other words, its
multiplicity is the rank of~$G$).  So, writing
$\mathfrak{g}_{\mathbb{C}}^\alpha := \{z \in \mathfrak{g}_{\mathbb{C}}
: [x,z] = i\alpha(x)\,z \text{~for all~} x\in\mathfrak{t}\}$ for the
weight space of $\alpha \in \Phi$ acting on $\mathfrak{g}_{\mathbb{C}}
:= \mathfrak{g} \otimes_{\mathbb{R}} \mathbb{C}$, we have the weight
space decomposition $\mathfrak{g}_{\mathbb{C}} =
\mathfrak{t}_{\mathbb{C}} \oplus \bigoplus_{\alpha\in\Phi}
\mathfrak{g}_{\mathbb{C}}^\alpha$ (compare
\cite[formula~(2.16)]{Knapp2002} and
\cite[theorem~6.38]{Kirillov2008}).

The set $\Phi$ of these roots is an abstract (reduced,
crystallographic) root system (\cite[theorem~19.2]{Bump2004} or
\cite[theorem~6.44]{Kirillov2008}), whose Weyl group is that which we
have already associated to~$G$ (cf.~\ref{par:maxtori}); it is
irreducible if and only if $G$ is simple.  Furthermore, this induces a
bijection between the isomorphism classes of semisimple compact Lie
algebras and root systems (\cite[corollary~7.55]{Kirillov2008} or
\cite[corollary~7.3]{Knapp2002}), or equivalently, isogeny classes of
semisimple compact Lie groups or isomorphism classes of semisimple
\emph{simply connected} compact Lie groups
(\cite[\S7.3]{FultonHarris2004}); we clarify in~\ref{par:liefundgrp}
below the classification of groups inside an isogeny class.

Note that, as a set, $\Phi$ depends not only on $G$ but on the choice
of the maximal torus~$T$ used to define the weights
(cf.~\ref{par:charweights}), or equivalently, not only on
$\mathfrak{g}$ but also on $\mathfrak{t}$: in fact, $\Phi$ is a
subset of the dual $\mathfrak{t}^*$ of~$\mathfrak{t}$; however, as an
abstract root system, it does not depend on this choice.

\smallskip

The following proposition follows immediately from what has already
been said:

\begin{proposition}\label{prop:adjcharvals}
If $G$ is a semisimple compact Lie group with rank~$n$, then the range
of values taken by the adjoint character $\ch_{\adjt}$ of~$G$ is
precisely the range of the function $n + \mathscr{F}_\Phi$ where
$\mathscr{F}_\Phi\colon x \mapsto \sum_{\alpha \in \Phi} e^{2\pi i
  \alpha(x)}$ is the Fourier transform of~$\Phi$.

More precisely, if $T$ is a maximal torus in~$G$ with Lie algebra
$\mathfrak{t}$, and $u \in T$ is written $\exp(2\pi x)$ for $x \in
\mathfrak{t}$, then $\ch_{\adjt}(u) = n + \mathscr{F}_\Phi(x)$ (and we
have pointed out in~\ref{par:maxtori} that each element $g$ of $G$ is
conjugate to an element $u$ of~$T$, which then obviously has
$\ch_{\adjt}(g) = \ch_{\adjt}(u)$).
\end{proposition}
\begin{proof}
As explained in~\ref{par:lieroots}, the weights of the adjoint
representation are the elements of $\Phi$ each with multiplicity~$1$,
and $0$ with multiplicity~$n$, i.e., $\ch_{\adjt}|_T = n\cdot
\mathbf{e}(0) + \sum_{\alpha\in\Phi} \mathbf{e}(\alpha)$, which is
precisely the statement of the second paragraph.
\end{proof}

\parapoint\label{par:liefundgrp}
We briefly clarify the relation between isogenous
(cf.~\ref{par:liegroupsvsalgs}) compact Lie groups in the semisimple
case (this subsection is required for completeness, but for the
purposes of this paper we care only about the simply connected
groups):

If $G$ is a semisimple compact Lie group, we have noted its root
system $\Phi$ can be defined directly from its Lie algebra
$\mathfrak{g}$ and that $\mathfrak{t}$ of a maximal torus $T$ of $G$
(which is the same as a maximal abelian subalgebra of~$\mathfrak{g}$,
cf.~\ref{par:maxtori}), namely as the set of nonzero
$\alpha\in\mathfrak{t}^*$ (where $\mathfrak{t}^*$ is the dual vector
space to~$\mathfrak{t}$) such that $\mathfrak{g}_{\mathbb{C}}^\alpha
:= \{z \in \mathfrak{g}
\otimes_{\mathbb{R}} \mathbb{C} : [x,z] = i\alpha(x)\,z \text{~for
  all~} x\in\mathfrak{t}\}$ is nontrivial (cf.~\ref{par:charweights}).
So the root lattice $Q := \mathbb{Z}\Phi$ and weight lattice $P :=
(\mathbb{Z}\Phi^\vee)^*$ defined in~\ref{par:rootsystems},
inside~$\mathfrak{t}^*$, are defined at the Lie algebra level: they
depend only on the isogeny class of~$G$.  We have $Q \subseteq
\Gamma^* \subseteq P$ or equivalently $Q^\vee \subseteq \Gamma
\subseteq P^\vee$ where $Q^\vee$ is the coroot and $P^\vee$ the
coweight lattice (the inclusion $Q \subseteq \Gamma^*$ follows from
the fact that the weights of any representation of $G$, as defined in
\ref{par:charweights}, belong to~$\Gamma^*$, and in particular the
roots belong to~$\Gamma^*$; the inclusion $Q^\vee \subseteq
\Gamma$ follows from the fact that the coroots can also be defined at
the Lie algebra level).

The classification of compact Lie groups having Lie
algebra~$\mathfrak{g}$ is then as follows: $G$~is uniquely defined by
giving the lattice $\Gamma$ satisfying $Q \subseteq \Gamma^* \subseteq
P$, or equivalently $Q^\vee \subseteq \Gamma \subseteq P^\vee$; and
conversely, for any such $\Gamma$, there exists a unique corresponding
Lie group~$G$; furthermore, the fundamental group of~$G$ is Abelian,
finite and canonically isomorphic to $\Gamma/Q^\vee$, and the center
$Z(G)$ of~$G$ is finite and canonically isomorphic to~$P^\vee/\Gamma$.
(\cite[chap.~4, \S3, 6\textsuperscript{o}, theorems
  9~and~10]{OnishchikVinberg1990}; see also
\cite[theorem~23.1]{Bump2004} and \cite[corollary~5.109]{Knapp2002}.)

(In particular, the universal covering of a semisimple compact Lie
group $G$ is still compact, and corresponds to taking $\Gamma$ equal
to the coroot lattice~$Q^\vee$.  At the other extreme, the centerless
group $G/Z(G)$ corresponds to taking $\Gamma$ equal to the coweight
lattice~$P^\vee$; this is also known as the ``adjoint'' form, because
it is the image of the adjoint representation $G \to
\mathit{GL}(\mathfrak{g})$.  The fundamental group of the adjoint
form, or equivalently the center of the universal covering, is defined
at the Lie algebra level, and is~$P^\vee/Q^\vee$.)

\begin{remark}\label{rmk:weylalcove}
If $G$ is a semisimple compact Lie group with maximal torus~$T =
\mathfrak{t}/\Gamma$, we have already pointed out in~\ref{par:maxtori}
that the set of conjugacy classes of~$G$ can be identified (as a set)
with $T/W$, where $W$ is the Weyl group.  Lifting to the Lie algebra
$\mathfrak{t}$ of~$T$, it can be identified with $\mathfrak{t}/(\Gamma
\rtimes W)$.  This point of view is particularly important when $G$ is
simply connected ($\Gamma$~is the coroot lattice) because then it can
be shown that the ``affine Weyl group'' $\Gamma \rtimes W$ is an
affine Coxeter group, having a fundamental domain, known as the
\emph{Weyl alcove}, which is the simplex whose vertices are $0$ and
the $\varpi_i^\vee/m_i$, where the $\varpi_i^\vee$ are the fundamental
coweights (cf.~\ref{par:posroots}) and $m_i$ are the coefficients of
the highest root (cf.~\ref{par:posroots}).  (See
\cite[chapter~11]{Kane2001}.)
\end{remark}

\parapoint\label{par:highestweight}
We now briefly review the classification of irreducible
representations of a semisimple compact Lie group as provided by
``highest weight theory''.

As explained in~\ref{par:charweights}, the weights of a
(finite-dimensional) representation $V$ of a semisimple compact Lie
group~$G$ (relative to the choice of a maximal torus $T \subseteq G$)
are the $\lambda \in \Gamma^*$ such that $V^\lambda := \{z\in V :
(\forall u\in T)\, u \cdot z = \mathbf{e}(\lambda)(u)\}$ is nonzero,
the multiplicity $m_\lambda$ being $\dim V^\lambda$.  Now fix a choice of
positive roots of~$G$ (cf.~\ref{par:posroots}): a \emph{highest
  weight} of~$V$ (or of its character,~$\chi$) is a weight $\lambda$
such that $\lambda+\alpha$ is not a weight for any positive root
$\alpha$; a \emph{dominant integral weight} (for~$G$) is a
$\lambda\in\Gamma^*$ belonging to the closed Weyl chamber, i.e., such
that $\lambda(\alpha^\vee) \geq 0$ for each simple (or equivalently,
positive) coroot~$\alpha^\vee$.

Highest weight theory tells us that (\cite[theorem~5.5]{Knapp2002},
\cite[\S8.3]{Kirillov2008} or \cite[chap.~4, \S3,
  7\textsuperscript{o}, theorem~11]{OnishchikVinberg1990} or
\cite[VI.1.7]{BroeckerDieck1985}):
\begin{itemize}
\item every irreducible representation of~$G$ has a unique highest
  weight, which is a dominant integral weight, and its multiplicity
  is~$1$;
\item if $V$ is an irreducible representation of~$G$ with highest
  weight~$\lambda$, then the set of weights of $V$ is the intersection
  of $\lambda+Q$, where $Q$ is the root lattice, and of the convex
  hull of the orbit of $\lambda$ under the Weyl group;
\item two irreducible representations of~$G$ are isomorphic if and
  only if they have the same highest weight (thus, we can speak of
  ``the'' irreducible representation with highest weight~$\lambda$);
\item every dominant integral weight is the highest weight of an
  irreducible representation of~$G$ (it is unique by the previous
  point);
\item if $V$ and $V'$ are irreducible representations of~$G$ with
  highest weights $\lambda$ and $\lambda'$ respectively, then
  $V\otimes V'$ has highest weight $\lambda+\lambda'$ and has a unique
  irreducible factor with that weight
  (\cite[VI.2.8]{BroeckerDieck1985}).
\end{itemize}

The highest weight of the adjoint representation is the highest root
($-\alpha_0$).

\parapoint\label{par:weylcharform} If $G$ is a semisimple compact Lie
group and $T$ a maximal torus of $G$, then the \emph{Weyl character
  formula} (\cite[\S8.5]{Kirillov2008},
\cite[theorems~5.75--5.77]{Knapp2002} or
\cite[VI.1.7]{BroeckerDieck1985}) expresses the value of the
irreducible character $\chi_\lambda$ with highest weight~$\lambda$
(i.e., the character of the irreducible representation with highest
weight~$\lambda$) as the ratio of two skew-$W$-invariant polynomials
on~$T$, namely
\[
\ch_\lambda = \frac
{\sum_{w\in W} \mathrm{sgn}(w)\,\mathbf{e}(w(\lambda+\rho))}
{\sum_{w\in W} \mathrm{sgn}(w)\,\mathbf{e}(w(\rho))}
\]
where $W$ is the Weyl group and $\mathrm{sgn} \colon W \to \{\pm 1\}$
the group homomorphism taking the value $-1$ on each reflection
$s_\alpha$ (i.e., $\mathrm{sgn}(w)$ is the determinant of $w$ acting
on the Lie algebra $\mathfrak{t}$ of~$T$); and the \emph{Weyl vector}
$\rho := \frac{1}{2}\sum_{\alpha\in\Phi_+} \alpha$ is half the sum of
the positive roots, which is also the sum $\sum_{i=1}^n \varpi_i$ of
the fundamental weights.  The denominator of the above expression can
be factored using the \emph{Weyl denominator formula}:
\[
\sum_{w\in W} \mathrm{sgn}(w)\,\mathbf{e}(w(\rho))
= \prod_{\alpha\in\Phi_+}(\mathbf{e}(\alpha/2) - \mathbf{e}(-\alpha/2)).
\]

\parapoint\label{par:fundchars}
Assuming that $G$ (still a semisimple compact Lie group) is simply
connected (so that $\Gamma^*$ is the weight lattice,
cf.~\ref{par:liefundgrp}), the irreducible representations having the
fundamental weights (cf.~\ref{par:posroots}) as highest weights are
known as \emph{fundamental representations}, and their characters as
the \emph{fundamental characters} of~$G$.

\parapoint\label{par:charring}
The \emph{character ring} of a compact Lie group~$G$ is the ring
generated by the characters of~$G$ (that is, the set of differences
$\chi_1 - \chi_2$ between two characters of~$G$) for pointwise sum and
product.  Equivalently, if we define a \emph{virtual representation}
of~$G$ to be the formal difference $V_1 \ominus V_2$ of two (finite
dimensional) representations, identifying $V_1 \ominus V_2$ with $V'_1
\ominus V'_2$ whenever $V_1 \oplus V'_2$ and $V'_1 \oplus V_2$ are
isomorphic (``Grothendieck ring'' construction), and if we define the
(virtual) character of $V_1 \ominus V_2$ to be $\chi_1 - \chi_2$ where
$\chi_i$ is the character of~$V_i$, the character ring can be defined
as the set of virtual representations of~$G$ with addition and
multiplication being defined as the direct sum and tensor product
(extended in the obvious fashion to virtual representations).

To put it differently, the character ring of~$G$ consists of
$\mathbb{Z}$-linear combinations of the irreducible characters (or
representations) of~$G$, the product being defined by the
decomposition into irreducibles of a product of characters (tensor
product of representations).  One can similarly define the character
$\mathbb{C}$-algebra as the set of $\mathbb{C}$-linear combinations of
the irreducible characters (or representations) of~$G$.

Highest weight theory implies that: (1)~for a semisimple compact Lie
group~$G$ with maximal torus~$T$, the character ring of~$G$ is simply
the invariant part under the Weyl group of the character ring of~$T$
(the latter being the group ring of~$\Gamma^*$,
cf.~\ref{par:torusrepns}), and (2)~when $G$ is, additionally, simply
connected, the character ring is isomorphic to the polynomial algebra,
with coefficients in $\mathbb{Z}$, over indeterminates corresponding
to the fundamental representations.  (\cite[VI.2.1 and
  VI.2.11]{BroeckerDieck1985}.)  The corresponding statements also
hold with complex coefficients instead of integers.

\begin{remark}
If $G$ is a semisimple simply connected compact Lie group with maximal
torus~$T = \mathfrak{t}/\Gamma$, then the character
$\mathbb{C}$-algebra of~$G$ can be identified (via restriction to~$T$)
with the set of $W$-invariant trigonometric polynomials on~$T$ (with
complex coefficients), where $W$ is the Weyl group, or, lifting to
$\mathfrak{t}$, of $W$-invariant (hence $\Gamma \rtimes W$-invariant)
combinations of the $\mathbf{e}(\lambda)$ for $\lambda \in \Gamma^*$.
Also note that such functions are entirely defined by their values on
the Weyl alcove (cf.~\ref{rmk:weylalcove}).

If we are mostly interested in the character values on~$T$ (they
determine those on~$G$ by~\ref{par:maxtori}), and in this paper we
are, the irreducible representations of~$G$ are something of a
needless complication: the character ring of a semisimple simply
connected compact Lie group has a $\mathbb{Z}$-basis consisting of the
sums $\sum_{\lambda \in W\lambda_0} \mathbf{e}(\lambda)$ for $\lambda$
ranging over an orbit of the Weyl group $W$ acting on~$\Gamma^*$.
\end{remark}

\parapoint\label{par:regularelements}
We have recalled in~\ref{par:maxtori} that every element $g$ of a
compact Lie group $G$ belongs to a maximal torus~$T$; when the torus
in question is \emph{unique}, the element $g$ is said to be
\emph{regular}.  Assuming that $G$ is semisimple, this is equivalent
(\cite[theorem~22.3(ii)]{Bump2004}) to saying that $g$ is not in the
kernel of any $\mathbf{e}(\alpha)$ for root $\alpha\in\Phi$
(cf.~\ref{par:torusrepns}).  Correspondingly, we say that an element
$x$ of the Lie algebra $\mathfrak{g}$ of~$G$ is regular when
$\mathbf{e}(x)$ is regular, i.e., when $\alpha(x) \not\in \mathbb{Z}$
for all $\alpha\in\Phi$ (this means that $x$ is represented by an
element in the interior of the Weyl alcove, cf.~\ref{rmk:weylalcove};
also compare \cite[V.7.8]{BroeckerDieck1985}).

The following fact is crucial to the proof given in
appendix~\ref{sec:serresproof}:

\begin{proposition}\label{prop:regularelementdifferentiels}
Let $G$ be a semisimple simply connected compact Lie group.  If $g \in
G$ and $T$ is a maximal torus containing $g$, then the following are
equivalent:
\begin{itemize}
\item the element $g$ is regular,
\item the differentials $d\ch_1,\ldots,d\ch_n$ of the fundamental
  characters of~$G$ (cf.~\ref{par:fundchars}) are independent
  at~$g$,
\item the differentials $d\ch_1|_T,\ldots,d\ch_n|_T$ of the
  fundamental characters of~$G$ restricted to~$T$ are independent
  at~$g$.
\end{itemize}
\end{proposition}

In a more general context, the equivalence of the first two statements
is due to Kostant (\cite[theorem~0.1]{Kostant1963}) and Steinberg
(\cite[theorem~8.1]{Steinberg1965}); however, since we are only
considering compact Lie group, every element $g$ belongs to a maximal
torus (i.e., is ``semisimple''), making the proof of the equivalence
considerably easier and giving the third statement as a byproduct (as
detailed in \cite[\S8.2--8.6]{Steinberg1965}).

\parapoint\label{par:parabolics}
We now briefly discuss how a subset of the nodes of the Dynkin diagram
of a semisimple compact Lie group defines a Lie subgroup with the
Dynkin diagram defined by the subset in question (i.e., the induced
subgraph).

So let $G$ be a semisimple compact Lie group, fix a maximal torus $T$
in~$G$, and let $\mathfrak{g},\mathfrak{t}$ be the corresponding Lie
algebras and $\Phi$ the root system of~$G$ (cf.~\ref{par:lieroots});
choose a system of simple roots $\alpha_1,\ldots,\alpha_n \in \Phi$
(where $n$ is the rank of~$G$).  Now if $I \subseteq \{1,\ldots,n\}$,
this defines a root system $\Phi_I \subseteq \Phi$, sometimes known as
the \emph{parabolic} subsystem associated to~$I$, namely the set
$\Phi_I := \Phi \cap \bigoplus_{i\in I} \mathbb{Z}\alpha_i = \Phi \cap
\bigoplus_{i\in I} \mathbb{R}\alpha_i$ ``generated by'' the $\alpha_i$
for $i\in I$ (see \cite[\S5.1]{Kane2001} for a discussion, or
\cite[\S12.1]{MalleTesterman2011}), so that its Dynkin diagram
consists of the nodes of that of $\Phi$ labeled by elements of~$I$.

We now fix such an $I$ and explain how to define a Lie subgroup of $G$
with root system~$\Psi := \Phi_I$.  See also
\cite[\S7.3--7.4]{Arvanitoyeorgos2003} for a more detailed and
pedagogical account of this construction.

For $\alpha \in \Phi$, let $\mathfrak{g}_{\mathbb{C}}^\alpha := \{z
\in \mathfrak{g}_{\mathbb{C}} : [x,z] = i\alpha(x)\,z \text{~for all~}
x\in\mathfrak{t}\}$ inside $\mathfrak{g}_{\mathbb{C}} := \mathfrak{g}
\otimes_{\mathbb{R}} \mathbb{C}$ be the corresponding weight space.
Then (see \cite[corollary~5.94]{Knapp2002}, or
\cite[proposition~12.6]{MalleTesterman2011} in a different context)
$\mathfrak{l}_{\mathbb{C}} := \mathfrak{t}_{\mathbb{C}} \oplus
\bigoplus_{\alpha\in\Psi} \mathfrak{g}_{\mathbb{C}}^\alpha$ is a
(complex) Lie subalgebra of $\mathfrak{g}_{\mathbb{C}}$ (sometimes
known as a ``parabolic Levi factor''), which can be further factored
as a Lie algebra direct sum (i.e., with trivial bracket between the
summands) of its center $\mathfrak{z}(\mathfrak{l}_{\mathbb{C}}) =
\bigcap_{\alpha\in\Psi} \ker\alpha \subseteq
\mathfrak{t}_{\mathbb{C}}$ and its semisimple subalgebra
$\mathfrak{l}_{\mathbb{C}}' = [\mathfrak{l}_{\mathbb{C}},
  \mathfrak{l}_{\mathbb{C}}] = \mathfrak{t}'_{\mathbb{C}} \oplus
\bigoplus_{\alpha\in\Psi} \mathfrak{g}_{\mathbb{C}}^\alpha$ where
$\mathfrak{t}'_{\mathbb{C}}$ is the complex subspace spanned by the
coroots $\alpha^\vee$ for~$\alpha\in\Psi$ (seen as elements of
$\mathfrak{t}_{\mathbb{C}}$).  Since $\mathfrak{l}_{\mathbb{C}}$ and
$\mathfrak{l}'_{\mathbb{C}}$ are stable under complex conjugation,
they define (real!) Lie subalgebras $\mathfrak{l} :=
\mathfrak{l}_{\mathbb{C}} \cap \mathfrak{g}$ and $\mathfrak{l}' :=
\mathfrak{l}'_{\mathbb{C}} \cap \mathfrak{g}$ of $\mathfrak{g}$, hence
compact Lie subgroups $L$, $L'$ of $G$ having these Lie algebras (see
\cite[theorem~5.114]{Knapp2002}).  Then $L'$ is a semisimple Lie
subgroup of~$G$ having root system~$\Psi$ and maximal torus $T'$ with
Lie algebra $\mathfrak{t}'$ (as for $L$, it is isogenous to the
product of $L'$ with a torus of rank $n-\#I$, so that it has the same
rank $n$ as~$G$).

(In fact, the only properties of $\Psi$ used above are that it is a
``closed'' subsystem of $\Phi$: see \cite[definition~13.2 and
  theorem~13.6]{MalleTesterman2011}.)

In appendix~\ref{sec:serresproof}, we will use the above construction
in the case where $I$ is the complement of a single node $\{i\}$ in
the Dynkin diagram.

\section{Serre's result on characters of compact Lie groups}
\label{sec:serresproof}

In this section, we prove the $\mathsf{E}_n$ cases of \cite[Theorem
  3']{Serre2004a}.  The proof in the $\mathsf{E}_7$ and~$\mathsf{E}_8$
cases has been kindly communicated to us by J.-P.~Serre and only
slightly adapted for symbolic computation with Sage (J.-P.~Serre was
able to perform the entire computation by hand and we have not
attempted to reproduce this feat; any errors in the following
expositions are, of course, entirely our own) and straightforwardly
extended to compute all critical values of the adjoint character; in
the $\mathsf{E}_6$ case, J.-P.~Serre referred us to a proof devised by
A.~Connes, which we do not follow here, preferring instead a
straightforward analogue of the $\mathsf{E}_7$ and~$\mathsf{E}_8$
cases, at the cost of considerably more computing power (the
$\mathsf{E}_6$ case does not seem doable by hand with the technique
presented below).

\begin{theorem}
\label{thm:serrestheorem}
If $G$ is a semisimple compact Lie group of type $\mathsf{E}_6$, $\mathsf{E}_7$ or~$\mathsf{E}_8$
respectively, and $\ch_{\adjt}$ its adjoint character
(cf.~\ref{par:lierepns}); then $\inf_{g\in G} \ch_{\adjt}(g)$ is equal
to $-3$, $-7$ or $-8$ respectively.
\end{theorem}

(The proof for $\mathsf{A}_n$ and $\mathsf{D}_n$ has been given in
\ref{thm:lowera}~and~\ref{thm:lowerd} respectively.)

In each case, we divide the proof in two steps: the \emph{reduction step}
and the \emph{computation step}, the second being itself
subdivided into an \emph{elimination substep} and a \emph{ruling-out substep}.

We will call $G$ the simply connected semisimple compact Lie group of
type $\mathsf{E}_6$, $\mathsf{E}_7$ or~$\mathsf{E}_8$ as the case may
be, $T$ its maximal torus (cf.~\ref{par:maxtori}) and $\mathfrak{t}$
the Lie algebra of the latter.

The reduction step uses the trick (\textdagger) explained below
(and based essentially on~\ref{prop:regularelementdifferentiels}) to
reduce the number of variables by observing that any critical point of
$\ch_{\ad}$ must lie on certain linear subspaces of $\mathfrak{t}$.
The \emph{computation step} then finds the critical values of some
polynomial function $h$ of several variables $x_1,\ldots,x_r$ by
elimination theory: there are slight variations in each of the cases
below, but broadly speaking, consider the ideal of
$\mathbb{C}[x_1,\ldots,x_r,y]$ generated by
$\frac{\partial h}{\partial x_i}$ and $y-h$ (defining the --- often
$0$-dimensional --- algebraic variety of critical points of $h$) and
use a Gr\"obner basis for some elimination order (that is, a monomial
order such that $y < x_1^{i_1} \cdots x_r^{i_r}$ for all
$i_1,\ldots,i_r$ not all zero) to find the projection of this variety
on the $y$ coordinate (represented by a polynomial in $y$ which one
factors to find the actual values, which for some currently mysterious
reason happen to be always rational); unfortunately, elimination
theory considers all \emph{complex} values of $x_1,\ldots,x_r$, so there are
many spurious values, and one must then consider each computed value,
or at least those that are smaller than the actual minimum, and rule
them out by showing that, for some reason, they cannot be realized for
real values $x_1,\ldots,x_r$ (generally by noticing that some other
element in the Gr\"obner basis does not have roots in the domain
considered).

Let us now explain the idea of the reduction step in more detail.  We
generally follow appendix~\ref{sec:liegroups} for notation: for
example, we call $\Phi$ the root system of~$G$.

The reduction trick (\textdagger) is as follows.  Suppose $z \in
\mathfrak{t}$ is a critical value of the adjoint character (which is a
fundamental character in each of $\mathsf{E}_6$, $\mathsf{E}_7$ and
$\mathsf{E}_8$), or more generally that it is a critical value of any
polynomial in the fundamental characters in which no fundamental
character appears more than once; then the result of Kostant and
Steinberg~\ref{prop:regularelementdifferentiels} implies that $z$ is
on a root hyperplane $\mathfrak{t}' = \{\alpha = m\}$ (where
$\alpha\in\Phi$ and $m\in\mathbb{Z}$); now the affine Weyl group
$\Gamma \rtimes W$ (cf.~\ref{rmk:weylalcove}) acts transitively on the
set of such root hyperplanes and preserves all character values, so we
can assume that the $z$ lies on the root hyperplane $\{\alpha_0 =
0\}$, where $-\alpha_0$ is the highest root, which (for
$\mathsf{D}_n$, $\mathsf{E}_6$, $\mathsf{E}_7$, $\mathsf{E}_8$) is
also one of the fundamental weight, say $\varpi_i$. (This $i$ can be
read by taking the \emph{extended} Dynkin diagram for $\Phi$: it is
the node to which the extender node attaches.)  In other words, the
coordinates of $z$ on the basis of coroots
$\alpha^\vee_1,\ldots,\alpha^\vee_n$ have a zero at coordinate~$i$
(the one which is measured by~$\varpi_i$): so $z$ lives, in fact, in
the hyperplane generated by $\alpha^\vee_j$ for $j\neq i$.  Now, as
explained in~\ref{par:parabolics}, the linear subspace $\mathfrak{t}'$
of $\mathfrak{t}$ generated by $\alpha^\vee_j$ for $j\in I$ (here $I
:= \{1,\ldots,n\} \setminus \{i\}$) is, in a natural way, the Lie
algebra of the maximal torus $T'$ of a semisimple Lie subgroup $G'$
of~$G$ (denoted $L'$ in~\ref{par:parabolics}), whose root system
$\Phi'$ has a Dynkin diagram obtained from that of $\Phi$ by keeping
only the roots labeled by an element of~$I$, so, in our case, by
deleting node~$i$.  So by restricting the character (initially the
adjoint character of~$G$) to $G'$, we are left with a character on a
Lie group $G'$ having a rank smaller by~$1$.  Of course, the character
restricted to the subgroup $G'$ in question is no longer the adjoint
character, but it can be expressed in terms of the fundamental
characters of $G'$ using standard tables of ``branching rules'' or a
computer program like
Sage\footnote{\url{http://doc.sagemath.org/html/en/thematic_tutorials/lie/branching_rules.html}}
(in fact, branching are typically given first for the restriction from
$G$ to an intermediate subgroup $G \supseteq G_1 \supseteq G'$,
maximal in~$G$, and described through the removal of the node $i$ from
the \emph{extended} Dynkin diagram of~$G$ by means of so-called
``Borel-de~Siebenthal theory: see \cite[\S8.10]{Wolff2011} or
\cite[chapter~12]{Kane2001}; the restriction from $G_1$ to $G$ is then
straightforward as $G$ is a factor of~$G_1$, and we will give both in
what follows).

In what follows, we number the nodes of the Dynkin diagrams as in
Bourbaki (cf.~figure~\ref{FigureDynkinDiagramNumbering}).  We write $\ch_i^G$ for the $i$-th fundamental
representation of the simple group $G$, and $\ch_{\ad}$ for the
adjoint representation: thus, $\ch_{\ad}^{\mathsf{E}_6} = \ch_2^{\mathsf{E}_6}$ and
$\ch_{\ad}^{\mathsf{E}_7} = \ch_1^{\mathsf{E}_7}$ and $\ch_{\ad}^{\mathsf{E}_8} = \ch_8^{\mathsf{E}_8}$.
More generally, we write $\ch_\lambda$ for the character with highest
weight $\lambda$ (written as a combination of fundamental weights
$\varpi_i$), so that $\ch_i$ is an abbreviation for $\ch_{\varpi_i}$.

\begin{theorem}\label{Theorem_CritValues_E6}
The set of critical values of $\ch_{\ad}^{\mathsf{E}_6}$ is $-3$, $-2$, $6$, $14$ and $78$.
We therefore have $\inf_{x} {\mathcal F}_{\mathsf{E}_6} = -9$ and $\chi({\mathsf E}_6) \geq 9$.
\end{theorem}

\proof If $z$ is a critical point of
$\ch_{\ad}^{\mathsf{E}_6} = \ch_2^{\mathsf{E}_6}$, i.e., if the differential
$d\ch_2^{\mathsf{E}_6}$ vanishes there, then by (\textdagger) we can assume
that $z$ belongs to (the Lie algebra $\mathfrak{t}$
of the maximal torus of) the Lie subgroup $\mathsf{A}_5$ defined by removing node $2$
from the Dynkin diagram of $\mathsf{E}_6$.

Now $\ch_2^{\mathsf{E}_6} |_{\mathsf{A}_5} = 2\ch_3^{\mathsf{A}_5} + \ch_1^{\mathsf{A}_5}\,\ch_5^{\mathsf{A}_5} +
2$.  In more detail, the branching rule for the maximal subgroup $\mathsf{A}_1
\times \mathsf{A}_5$ of $\mathsf{E}_6$ gives: $\ch_2^{\mathsf{E}_6} |_{\mathsf{A}_1 \times \mathsf{A}_5} =
\ch_{\varpi_1}^{\mathsf{A}_1} \ch_{\varpi_3}^{\mathsf{A}_5} + \ch_{\varpi_1 +
  \varpi_5}^{\mathsf{A}_5} + \ch_{2\varpi_1}^{\mathsf{A}_1}$ as witnessed by Sage:

\begin{verbatim}
sage: E6 = WeylCharacterRing("E6", style="coroots")
sage: br = branching_rule(E6, "A1xA5","extended")
sage: br.branch(E6(E6.fundamental_weights()[2]))
A1xA5(1,0,0,1,0,0) + A1xA5(2,0,0,0,0,0) + A1xA5(0,1,0,0,0,1)
\end{verbatim}

We then observe that
\[
\ch_{\varpi_1 + \varpi_5}^{\mathsf{A}_5} = \ch_1^{\mathsf{A}_5}
\ch_5^{\mathsf{A}_5} - 1\]
and that
\[
\ch_{2\varpi_1}^{\mathsf{A}_1} = (\ch_1^{\mathsf{A}_1})^2 -
1,\]
so that
\[
\ch_2^{\mathsf{E}_6} |_{\mathsf{A}_1 \times \mathsf{A}_5} = \ch_1^{\mathsf{A}_1} \ch_3^{\mathsf{A}_5}
+ (\ch_1^{\mathsf{A}_5}\, \ch_5^{\mathsf{A}_5} - 1) + ((\ch_1^{\mathsf{A}_1})^2 - 1).
\]
Evaluating at the identity of $\mathsf{A}_1$ (where $\ch_1^{\mathsf{A}_1} = 2$), we get
\[
\ch_2^{\mathsf{E}_6} |_{\mathsf{A}_5} = 2 \ch_3^{\mathsf{A}_5} + \ch_1^{\mathsf{A}_5}\,
\ch_5^{\mathsf{A}_5} + 2
\]
as announced.

Now the fundamental characters $\ch_i$ of $\mathsf{A}_5$ are the elementary
symmetric functions $\sigma_i$ of six variables $u_0,\ldots,u_5$
ranging over the unit circle $\mathbb{U} := \{u \in \mathbb{C} :
|u|=1\}$ and constrained by $u_0 u_1 u_2 u_3 u_4 u_5 = 1$ (the
eigenvalues of the element of $\mathit{SU}_6$).  This means
that we are to compute the critical values of $h := 2\sigma_3 +
\sigma_1 \sigma_5 + 2$ over $\{(u_0,\ldots,u_5)\in\mathbb{U}^6 : u_0
u_1 u_2 u_3 u_4 u_5 = 1\}$ (which is the maximal torus of $\mathsf{A}_5$).

This concludes the reduction step (which is simpler in the case of
$\mathsf{E}_6$ than for $\mathsf{E}_7,\mathsf{E}_8$), and we now
proceed to the computation step (which, compared to
$\mathsf{E}_7,\mathsf{E}_8$, has fewer cases to consider, but is
computationally more challenging).

By elimination theory, we can compute the critical values for
$u_0,\ldots,u_5$ ranging over $\mathbb{C}^6$ subject to $u_0 u_1 u_2
u_3 u_4 u_5 = 1$: consider the ideal of $\mathbb{C}[u_0,\ldots,u_5,y]$
generated by $u_i \frac{\partial h}{\partial u_i} - u_0 \frac{\partial
  h}{\partial u_0}$ for $i=1,\ldots,5$ (because saying that $dh$ is
proportional to $d(u_0\cdots u_5)$ means $u_i \frac{\partial
  h}{\partial u_i} = u_0 \frac{\partial h}{\partial u_0}$) and also $y
- h$; and perform elimination of the variables $u_0,\ldots,u_5$ (by
computing a Gr\"obner basis for a monomial order for which $y <
u_0^{i_0} \cdots u_5^{i_5}$ for any $i_0,\ldots,i_5$ not all zero) in
this ideal to obtain the projection on the $y$ coordinate of the
critical points.

By computing the Gr\"obner basis of the corresponding ideal we obtain
that the resulting set of possible critical values is $-66$, $-3$, $-2$,
$6$, $14$ and $78$.

Now the critical value $-66$ cannot be attained on
\[
\{(u_0,\ldots,u_5)\in\mathbb{U}^6 : u_0 u_1 u_2 u_3 u_4 u_5 = 1\}.
\]
We add the inequality $y+66$ to the ideal and recompute a Gr\"obner
basis. In this basis we have $\sigma_5^3 = -6^3$.

And this is impossible because $\sigma_5$ is the sum of the
$u_i^{-1}$, which can only take the value $6$ in absolute value
provided all the $u_i$ are equal to one and the same $6$-th root of
unity $\zeta$, in which case $\sigma_1 = 6\zeta$ and $\sigma_5 =
6\zeta^{-1}$ and $\sigma_3 = 20\zeta^3$ and by checking the possible
$\zeta$ one notices that $h = 2\sigma_3 + \sigma_1\,\sigma_5 + 2$ does
not, in fact, take the value $-66$.

The value $-3$, on the other hand, is attained, namely when
three of the $u_i$ are equal to one primitive cube root of unity and
the other three are equal to the other. So it is its minimum and so
a critical value.

By further Gr\"obner basis computation we obtain the $78$ is attained
only at $u_i=1$. We also obtained that the critical value $14$ is
attained only with four $u_i$ set at $-1$ and two $u_i$ set at $1$.
The critical value $6$ is attained only by fixing two $u_i$ at $1$, two
at $e^{i 2\pi/3}$ and two at $e^{-2i \pi/3}$.
The critical value $-2$ corresponds to a manifold of dimension $2$.
In this manifold is the point with four $u_i$ set at $1$ and two $u_i$ set at $-1$.

From the formula
${\mathcal F}_{\mathsf{E}_6}(x) = \ch_{\ad}^{\mathsf{E}_6}(x) - 6$ we
get $\inf_{x} {\mathcal F}_{\mathsf{E}_6}(x) = -9$ and then using
Corollary~\ref{cor:Hoffman}
$\chi(\mathsf{E}_6) \geq 1 - (-9/72)^{-1} = 9$. \qed

\begin{theorem}\label{Theorem_CritValues_E8}
The set of critical values of $\ch_{\ad}^{\mathsf{E}_8}$ is $-8$, $-4$, $-\frac{104}{27}$, $-\frac{57}{16}$, $-3$, $-2$, $0$, $5$, $24$, $248$.
We therefore have $\inf_{x} {\mathcal F}_{\mathsf{E}_8} = -16$ and $\chi({\mathsf E}_8) \geq 16$.
\end{theorem}

\proof If $z$ is a critical point of
$\ch_{\ad}^{\mathsf{E}_8} = \ch_8^{\mathsf{E}_8}$, i.e., if the differential
$d\ch_8^{\mathsf{E}_8}$ vanishes there, then by (\textdagger) we can assume
that $z$ belongs to (the Lie algebra $\mathfrak{t}$
of the maximal torus of) the Lie subgroup $\mathsf{E}_7$ defined by removing node $8$
from the Dynkin diagram of $\mathsf{E}_8$.

Now
\[
\ch_8^{\mathsf{E}_8} |_{\mathsf{E}_7} = \ch_1^{\mathsf{E}_7} +
2\ch_7^{\mathsf{E}_7} + 3 \text{ (from
$\ch_8^{\mathsf{E}_8} |_{\mathsf{A}_1 \times \mathsf{E}_7} = \ch_1^{\mathsf{E}_7} +
\ch_1^{\mathsf{A}_1}\,\ch_7^{\mathsf{E}_7} +
((\ch_1^{\mathsf{A}_1})^2-1)$).}
\]

Now since neither $\ch_1^{\mathsf{E}_7}$ nor $\ch_7^{\mathsf{E}_7}$ appear more than
once (or with any exponent) in $\ch_1^{\mathsf{E}_7} + 2\ch_7^{\mathsf{E}_7} + 3$, we
can apply (\textdagger) again: at a point $z$ where the differential
of this expression vanishes, the differentials of the fundamental
characters are not independent, so $z$ belongs to (the Lie algebra $\mathfrak{t}$
of the maximal torus of) the Lie subgroup $\mathsf{D}_6$ defined
by removing node $1$ from the Dynkin diagram of $\mathsf{E}_7$.

We have
\[
\ch_1^{\mathsf{E}_7} |_{\mathsf{D}_6} = \ch_2^{\mathsf{D}_6} +
2\ch_5^{\mathsf{D}_6} + 3 \text{ (from
$\ch_1^{\mathsf{E}_7} |_{\mathsf{A}_1\times \mathsf{D}_6} = \ch_2^{\mathsf{D}_6} +
\ch_1^{\mathsf{A}_1}\,\ch_5^{\mathsf{D}_6} +
((\ch_1^{\mathsf{A}_1})^2-1)$)}
\]
and
\[
\ch_7^{\mathsf{E}_7}
|_{\mathsf{D}_6} = \ch_6^{\mathsf{D}_6} + 2\ch_1^{\mathsf{D}_6} \text{ (from $\ch_7^{\mathsf{E}_7} |_{\mathsf{A}_1\times
  \mathsf{D}_6} = \ch_6^{\mathsf{D}_6} +
\ch_1^{\mathsf{A}_1}\,\ch_1^{\mathsf{D}_6}$),}
\]
giving:
\[
\ch_{\ad}^{\mathsf{E}_8}|_{\mathsf{D}_6} = \ch_2^{\mathsf{D}_6} + 2\ch_5^{\mathsf{D}_6} + 2\ch_6^{\mathsf{D}_6} +
4\ch_1^{\mathsf{D}_6} + 6.
\]

Again, there are no multiple occurrences of the various $\ch_i^{\mathsf{D}_6}$,
so we can apply (\textdagger) one more time: at a point $z$ where the
differential of this expression vanishes, the differentials of the
fundamental characters are not independent, so $z$ belongs to (the Lie algebra $\mathfrak{t}$
of the maximal torus of) the Lie subgroup
$\mathsf{D}_4\times \mathsf{A}_1$ defined by removing node $2$ from the Dynkin diagram
of $\mathsf{D}_6$.

Here the branching gets more complicated: We use
\begin{verbatim}
sage: WeylCharacterRing("D6").maximal_subgroups() 
\end{verbatim}
within Sage to find the correct rule for branching to
$\mathsf{A}_1 \times \mathsf{A}_1 \times \mathsf{D}_4$, and in
principle the two $\mathsf{A}_1$ factors are not symmetric (although
in the end it turns out that they are, up to a symmetry of
$\mathsf{D}_4$) so one must use \verb=br.describe()= to chop off the
correct $\mathsf{A}_1$ factor (call it $\mathsf{A}_1^\circ$ in what
follows).  We find:
\begin{itemize}
\item $\ch_2^{\mathsf{D}_6} |_{\mathsf{A}_1\times \mathsf{D}_4} = \ch_2^{\mathsf{D}_4} + 2 \ch_1^{\mathsf{A}_1}\,
  \ch_1^{\mathsf{D}_4} + ((\ch_1^{\mathsf{A}_1})^2-1) + 3$ (from $\ch_2^{\mathsf{D}_6}
  |_{\mathsf{A}_1^\circ\times \mathsf{A}_1\times \mathsf{D}_4} = \ch_2^{\mathsf{D}_4} +
  \ch_1^{\mathsf{A}_1^\circ}\, \ch_1^{\mathsf{A}_1}\, \ch_1^{\mathsf{D}_4} + ((\ch_1^{\mathsf{A}_1})^2-1)
  + ((\ch_1^{\mathsf{A}_1^\circ})^2-1)$).
\item $\ch_5^{\mathsf{D}_6} |_{\mathsf{A}_1\times \mathsf{D}_4} = \ch_1^{\mathsf{A}_1}\, \ch_4^{\mathsf{D}_4} + 2
  \ch_3^{\mathsf{D}_4}$ (from $\ch_5^{\mathsf{D}_6} |_{\mathsf{A}_1^\circ\times \mathsf{A}_1\times \mathsf{D}_4} =
  \ch_1^{\mathsf{A}_1}\, \ch_4^{\mathsf{D}_4} + \ch_1^{\mathsf{A}_1^\circ}\, \ch_3^{\mathsf{D}_4}$).
\item $\ch_6^{\mathsf{D}_6} |_{\mathsf{A}_1\times \mathsf{D}_4} = \ch_1^{\mathsf{A}_1}\, \ch_3^{\mathsf{D}_4} + 2
  \ch_4^{\mathsf{D}_4}$ (from $\ch_5^{\mathsf{D}_6} |_{\mathsf{A}_1^\circ\times \mathsf{A}_1\times \mathsf{D}_4} =
  \ch_1^{\mathsf{A}_1}\, \ch_3^{\mathsf{D}_4} + \ch_1^{\mathsf{A}_1^\circ}\, \ch_4^{\mathsf{D}_4}$).
\item $\ch_1^{\mathsf{D}_6} |_{\mathsf{A}_1\times \mathsf{D}_4} = \ch_1^{\mathsf{D}_4} + 2 \ch_1^{\mathsf{A}_1}$
  (from $\ch_1^{\mathsf{D}_6} |_{\mathsf{A}_1^\circ\times \mathsf{A}_1\times \mathsf{D}_4} = \ch_1^{\mathsf{D}_4} +
  \ch_1^{\mathsf{A}_1^\circ}\,\ch_1^{\mathsf{A}_1}$).
\end{itemize}
Finally, we get:
\[
\begin{split}
\ch_{\ad}^{\mathsf{E}_8}|_{\mathsf{A}_1\times \mathsf{D}_4} = & \ch_2^{\mathsf{D}_4} + 2
\ch_1^{\mathsf{A}_1}\, \ch_1^{\mathsf{D}_4} + 2 \ch_1^{\mathsf{A}_1}\,\ch_4^{\mathsf{D}_4} + 2
\ch_1^{\mathsf{A}_1}\, \ch_3^{\mathsf{D}_4} \penalty-100 + 4\ch_1^{\mathsf{D}_4} + 4\ch_4^{\mathsf{D}_4} +
4\ch_3^{\mathsf{D}_4} \penalty-100 \\
& + (\ch_1^{\mathsf{A}_1})^2 +
8(\ch_1^{\mathsf{A}_1}) + 8.
\end{split}
\]

We can now apply the reduction trick (\textdagger) one last time, for the $\mathsf{D}_4$
factor: since $\ch_1^{\mathsf{A}_1}$ is obviously independent from the
$\ch_i^{\mathsf{D}_4}$, at a point $z$ where the differential of the expression
$\ch_{\ad}^{\mathsf{E}_8}|_{\mathsf{A}_1\times \mathsf{D}_4}$ above vanishes, the differentials
of the fundamental characters $\ch_i^{\mathsf{D}_4}$ are not independent, so
$z$ belongs to (the Lie algebra $\mathfrak{t}$
of the maximal torus of) the Lie subgroup $\mathsf{A}_1 \times (\mathsf{A}_1)^3$ defined by removing node $2$
from the Dynkin diagram of $\mathsf{D}_4$.

To compute this restriction, first examine the restriction of $\mathsf{D}_4$ to
the maximal subgroup $\mathsf{A}_1 \times \mathsf{A}_1 \times \mathsf{A}_1 \times \mathsf{A}_1$ of $\mathsf{D}_4$
(seen by extending the Dynkin diagram of $\mathsf{D}_4$ and removing the node
connected to the four others): if we call $t_1,\ldots,t_4$ the
(single) fundamental characters of the various $\mathsf{A}_1$ factors, numbered
in the same way as the nodes of the extended diagram of $\mathsf{D}_4$ from
which they come (except that $t_2$ comes from the extending node),
then $\ch_1^{\mathsf{D}_4} |_{(\mathsf{A}_1)^4} = t_1 t_2 + t_3 t_4$ and $\ch_3^{\mathsf{D}_4}
|_{(\mathsf{A}_1)^4} = t_1 t_4 + t_2 t_3$ and $\ch_4^{\mathsf{D}_4} |_{(\mathsf{A}_1)^4} = t_1
t_3 + t_2 t_4$ and finally
\[
\ch_2^{\mathsf{D}_4} |_{(\mathsf{A}_1)^4} = t_1 t_2 t_3 t_4
+ t_1^2 + t_2^2 + t_3^2 + t_4^2 - 4.
\]

Finally, restricting $\ch_{\ad}^{\mathsf{E}_8}$ to $\mathsf{A}_1 \times (\mathsf{A}_1)^3$ (the
first $\mathsf{A}_1$ factor being the factor $\mathsf{A}_1$ in $\mathsf{A}_1\times \mathsf{D}_4$ earlier
and the other three coming from nodes $1,3,4$ of $\mathsf{D}_4$ as described in
the previous paragraphs), we have
\[
\ch_{\ad}^{\mathsf{E}_8}|_{\mathsf{A}_1\times
  (\mathsf{A}_1)^3} = 2\sigma_3 + \sigma_1^2 + 2(s+1)\sigma_2 + 4(s+2)\sigma_1
+ s^2 + 8s + 4,
\]
where $s$ is the fundamental character from the first $\mathsf{A}_1$
factor and $\sigma_i$ are the elementary symmetric functions in the
fundamental characters $t_1,t_3,t_4$ of the three other $\mathsf{A}_1$
factors.

We now need to find the critical values of this function
\[
h =
2\sigma_3 + \sigma_1^2 + 2(s+1)\sigma_2 + 4(s+2)\sigma_1 + s^2 + 8s +
4. 
\]
There is one subtlety, however: ``critical'' means that for each
$t_i$, as well as for $s$, we either have $\frac{\partial h}{\partial t_i} = 0$
(resp. $\frac{\partial h}{\partial s} = 0$) \emph{or} $t_i = \pm 2$
(resp. $s = \pm 2$).  Indeed, each $t_i$ (as well as $s$) is
a character of $\mathsf{A}_1$, so it is $u + u^{-1}$ for the two eigenvalues
$u,u^{-1}$ of the element of $\mathit{SU}_2$ in question, so the
critical values of $t_i$ itself are $\pm 2$.

This concludes the \emph{reduction step} for $\mathsf{E}_8$.  The
\emph{computation step} is then to use elimination theory, in each
possible case depending on how many of the $t_i$ satisfy
$\frac{\partial h}{\partial t_i} = 0$ and how many satisfy $t_i = 2$
and $t_i = -2$, and similarly for $s$, to find the corresponding
values of $h = 2\sigma_3 + \sigma_1^2 + 2(s+1)\sigma_2 +
4(s+2)\sigma_1 + s^2 + 8s + 4$.

For this, we must consider $10 \times 3 = 30$ cases according to
constraints placed on the $t_i$ (which can be set equal to $+2$ or to
$-2$ or to satisfy $\frac{\partial h}{\partial t_i} = 0$, which we
denote as ``$t_i = \partial$'' for short) and on $s$ (similarly $s=+2$
or $s=-2$ or $s=\partial$).  In each case, we consider the ideal of
$\mathbb{C}[s,t_1,t_3,t_4,y]$ generated by the $t_i - 2$ or $t_i + 2$
or $\frac{\partial h}{\partial t_i}$ as the case may be, and similarly
for $s$, and also $y - h$; and perform elimination of the variables
$s,t_1,t_3,t_4$ (by computing a Gr\"obner basis for a monomial order for
which $y < s^j t_1^{i_1} t_3^{i_3} t_4^{i_4}$ for any $j,i_1,i_3,i_4$
not all zero) in this ideal to obtain the projection on the $y$
coordinate of the critical points.

By considering all cases we find that the set of possible critical values
is $-652$, $-27$, $-12$, $-\frac{64}{7}$, $-8$, $-4$, $-\frac{104}{27}$, $-\frac{57}{16}$, $-3$,
$-2$, $0$, $5$, $24$ and $248$.
For each such value and each possible critical value we compute the
manifold which turns out to be always $0$-dimensional.
The points of those manifolds can be enumerated and we obtain the list of critical
values by keeping only the values for which at least one of the point has
$\vert t_1\vert$, $\vert t_3\vert$, $\vert t_4\vert$, $\vert s\vert\leq 2$.

From the formula ${\mathcal F}_{\mathsf{E}_8}(x) = \ch_{\ad}^{\mathsf{E}_8}(x) - 8$
we get $\inf_{x} {\mathcal F}_{\mathsf{E}_8}(x) = -16$ and then using Corollary~\ref{cor:Hoffman}
$\chi(\mathsf{E}_8) \geq 1 - (-16/240)^{-1} = 16$. \qed

\begin{theorem}
\label{Theorem_CritValues_E7}
The set of critical values of $\ch_{\ad}^{\mathsf{E}_7}$ is $-7$, $-3$, $-2$, $1$, $\frac{17}{5}$, $5$, $25$, $133$.
We therefore have $\inf_{x} {\mathcal F}_{\mathsf{E}_7} = -14$ and $\chi({\mathsf E}_7) \geq 10$.
\end{theorem}
\proof We have $\ch_{\ad}^{\mathsf{E}_7} = \ch_1^{\mathsf{E}_7}$
and all the reduction step has already been explained above in
the $\mathsf{E}_8$ case: we get
\[
\ch_{\ad}^{\mathsf{E}_7}|_{\mathsf{A}_1\times \mathsf{D}_4} = \ch_2^{\mathsf{D}_4}
+ 2 \ch_1^{\mathsf{A}_1}\, \ch_1^{\mathsf{D}_4} + 2 \ch_1^{\mathsf{A}_1}\,\ch_4^{\mathsf{D}_4}
\penalty-100 + 4\ch_3^{\mathsf{D}_4} + (\ch_1^{\mathsf{A}_1})^2
\penalty-100 + 5,
\] 
so that
\[
\ch_{\ad}^{\mathsf{E}_7}|_{\mathsf{A}_1\times (\mathsf{A}_1)^4} = 2(\sigma_3-\sigma_2) +
\sigma_1^2 + 2 s (t_1+t_4)(t_3 + 2) + 4(t_1 t_4 + 2 t_3) + s^2 + 5.
\]

The computation step is then similar to $\mathsf{E}_8$, except there is now
less symmetry between the $t_i$ (one can only exchange $t_1$
and $t_4$): one must therefore distinguish $3\times 3\times 6$ cases.

By computing a Gr\"obner basis we obtain that the set of possible critical
values is $-191$, $-11$, $-35/4$, $-7$, $-3$, $-2$, $1$, $\frac{17}{5}$, $5$, $25$ and $133$.
For each case and critical value we compute the corresponding manifold
and its complex points.
In each case except one the manifold is $0$-dimensional and for two
$0$-dimensional cases the computation of the points does not finish.
Those $3$ problematic cases occur for the value $-2$.

For the other cases we compute the points and keep only the values for
which one of the points has $\vert t_1\vert$, $\vert t_3 \vert$,
$\vert t_4\vert$, $\vert s\vert \leq 2$. It turns out that the value $-2$
is attained by one of those points and so the $3$ problematic cases
do not prevent us from concluding that $-2$ is a critical value.

From the formula ${\mathcal F}_{\mathsf{E}_7}(x) = \ch_{\ad}^{\mathsf{E}_7}(x) - 7$
we get $\inf_{x} {\mathcal F}_{\mathsf{E}_7}(x) = -14$ and then using Corollary~\ref{cor:Hoffman}
$\chi(\mathsf{E}_7) \geq 1 - (-14/126)^{-1} = 10$. \qed

\bigskip

One possible extension of this work could be to consider the non-simply laced diagrams,
that is $\mathsf{B}_n$, $\mathsf{C}_n$, $\mathsf{F}_4$ and $\mathsf{G}_2$.

\bigskip

\end{document}